\documentclass[preprint]{imsart}
\usepackage{amsmath,amsthm,amssymb,graphicx,enumerate,natbib,stmaryrd}

\RequirePackage[OT1]{fontenc}
\RequirePackage{amsthm,amsmath}

\newtheorem{theorem}{Theorem}[section]
\newtheorem{lemma}[theorem]{Lemma}

\newtheorem{definition}[theorem]{Definition}
\newtheorem{assumption}[theorem]{Assumption}
\theoremstyle{remark}
\newtheorem{remark}[theorem]{Remark}
\newtheorem{example}[theorem]{Example}
\numberwithin{equation}{section}
\begin{document}

\begin{frontmatter}
\title{A Functional Large and Moderate Deviation Principle for Infinitely Divisible Processes Driven by Null-Recurrent Markov Chains}
\runtitle{LDP for ID Processes}

\begin{aug}
\author{\fnms{Souvik} \snm{Ghosh}\ead[label=e1]{ghosh@stat.columbia.edu}}

\runauthor{S. Ghosh}

\affiliation{Columbia University}

\address{Department of Statistics\\
Columbia University \\
New York, NY 10027\\
\printead{e1}}

\end{aug}

\begin{abstract}
Suppose  $ E$ is a space with a null-recurrent Markov kernel $ P$.  Furthermore, suppose there are infinite particles with variable weights on $ E$ performing a random walk following $ P$. Let $ X_{ t}$ be a weighted functional of the position of particles at time $ t$.  Under some conditions on the initial distribution of the particles the process $ (X_{ t})$ is stationary over time.  Non-Gaussian infinitely divisible (ID) distributions turn out to be natural candidates for the initial distribution and then the process $ (X_{ t})$ is ID.  We prove  a functional large and moderate deviation principle for the partial sums of the process $ (X_{ t})$. The recurrence of the Markov Kernel $ P$ induces long memory in the process $ (X_{ t})$ and that is reflected in the large deviation principle. It has been observed in certain short memory processes that  the large deviation principle is very similar to that of an i.i.d. sequence. Whereas, if the process is long range dependent the large deviations change dramatically. We show that  a similar phenomenon is observed for  infinitely divisible processes driven by Markov chains. 
\end{abstract}



\end{frontmatter}

\begin{section}{Introduction}
A $ \mathbb{R}^{ d}$-valued random variable $ X$ is said to have an infinitely divisible (ID) distribution if for any integer $ k>0$  there exists independent and identically distributed (i.i.d.) random variables $ W_{ 1}^{ ( k)},\ldots , W_{ k}^{ ( k)}$ such that 
\[ 
	X \stackrel{ d}{ =}W_{ 1}^{ ( k)}+\cdots+W_{ k}^{ ( k)},
\]
where $ \stackrel{ d}{= } $ denotes equality in distribution.
The class of infinitely divisible distributions is  broad and includes,  for example, the Gaussian, Cauchy, stable, compound Poisson, negative binomial, gamma, Student's t, F, Gumbel, log-normal, generalized Pareto and the logistic distribution among others. \cite{sato:1999} gives an excellent exposure to this topic. 

A process $ (   X_{ n},n\in \mathbb{Z}  )$ is said to be an infinitely divisible process if every finite dimensional marginal of the process is infinitely divisible. This is again a broad class of processes, with the most popular being the  L\'evy process.  
\cite{maruyama:1970} characterized the structure of infinitely divisible processes and since then several articles have contributed to the better understanding of these processes:   \cite{rajput1989spectral}, \cite{rosinski:1990}, \cite{rosinskiZak:1997}, \cite{marcus2005continuity} and \cite{roy2007ergodic} are  a few to mention. \cite{rosinski:2007} is a very instructive lecture note in this regard. 

  We consider  a class of long range dependent stationary ID processes which arise naturally in infinite particle systems.  Suppose $ E$ is a space with a null-recurrent Markov kernel $ P$. Also suppose there are infinitely many particles indexed by $ i\in \mathbb{N}$, where the $ i$th particle has weight $ W_{ i}$ and its position at time $ t$ is $ Z_{ i,t}$. Define the random variable $ X_{ t}$ as the weighted functional of the position of the particles:
  \begin{equation}\label{eq:vaguedescofX} 
 	X_{ t}=\sum_{ i\ge 1} W_{ i}f(Z_{ i,t}) \ \ \ \ \mbox{ for all } t\ge 0.   
\end{equation}
Under certain conditions on the function $ f:E\to \mathbb{R}$ and the weights $ W_{ i}$ the random variable $ X_{ t}$ is well defined. For a suitable initial distribution of the particles, the process $ (X_{ t})$ is a stationary and ergodic  ID process with  light tailed marginals, i.e., 
\begin{equation}\label{eq:mgffinite}
 	E  \big[ \exp ( \lambda  X_{ 0}) \big] < \infty, \mbox{ for all }  \lambda \in \mathbb{R}.
\end{equation} Infinitely divisible distributions arise as limits of such systems and hence are natural candidates for the initial distribution. \cite{liggett1988systems} and \cite{hoffman1995convergence} discusses convergence of system independent particles with unit mass to Poisson random measures. Invariance properties of similar particle systems have been studied by \citep[p.404-407]{doob1953stochastic} and \cite{spitzer1977stochastic}.

 We study large deviation principle for normalized partial sums of the process $ (X_{ t})$.  A sequence of probability measures $(\mu_n)$ on the Borel subsets of a topological space is said to satisfy the \emph{large deviation principle}, or LDP, with \emph{speed} $b_n$, and \emph{rate function}  $I(\cdot)$, if for any Borel set $A$, 
\begin{equation}\label{ldpdef}
-\inf_{x\in A^{\circ}}I(x)\le \liminf_{n\rightarrow \infty}\frac{1}{b_n}\log \mu_n(A)\le \limsup_{n\rightarrow \infty}\frac{1}{b_n}\log \mu_n(A)\le -\inf_{x\in \bar{A}}I(x), 
\end{equation}
where $ A^{\circ}$ and $\bar{A}$ are the interior and closure of $A$, respectively. A rate function is a non-negative lower semi-continuous function. Recall  that a function is said to be lower semicontinuous if its level sets are closed. A rate function is said to be  \emph{good} if it has compact level sets.  We refer to \cite{varadhan:1984}, \cite{deuschel:stroock:1989} or \cite{dembo:zeitouni:1998}  for a detailed treatment of  large deviations. 

 We  take $ \mu_{ n}$ to be the law of the random variable $ a_{ n}^{ -1}(X_{ 1}+\cdots+X_{ n})$ or its functional counterpart. 
The sequence $ (a_{ n})$ is growing faster than the rate required for a non-degenerate weak limit.
 It is standard to make a distinction between ``proper large deviations'' and ``moderate deviations''. Moderate deviations is the regime where the normalizing sequence is faster than the rate required for  weak convergence but slow enough so that the rate function reflects the non-degenerate weak limit. The proper large deviations is the regime where this effect in the rate function disappears. For  i.i.d. sequences $X_1,X_2,\ldots$ the proper large deviations regime corresponds to the linear growth of the normalizing sequence and then the speed sequence is also linear; see Theorem 2.2.3 (Cram\'er's Theorem) in \cite{dembo:zeitouni:1998}. In that situation the moderate deviation regime is the one where $ (a_{ n})$ grows slower than a linear rate but faster than $ \sqrt{n}$. The same remains true for certain short memory processes but this changes drastically  for certain long memory processes. \cite{ghosh:samorodnitsky:2009} showed that for long range dependent  moving average processes the natural boundary for proper large deviation is not the linear normalizing sequence but the speed sequence is linear.  We show examples of long memory processes where the proper large deviation  regime has linear normalizing sequence and a speed sequence is growing slower than the  linear rate.

 The breadth of the class of ID processes makes it very difficult to develop a general theory of large deviations.  LDP is known for some special classes of ID processes. \cite{donsker:varadhan:1985} proved LDP for the empirical measures of a Gaussian process with continuous spectral density. \cite{ghosh:samorodnitsky:2009} proved functional large deviation principle for short and long memory moving averages which gives an almost complete picture of LDP for Gaussian processes.  \cite{deAcosta:1994} proved functional large deviations for the L\'evy processes taking values in some Banach space under certain integrability assumptions, which in the finite dimensional Euclidean space, is identical to \eqref{eq:mgffinite}.

This article is arranged as follows. We discuss the necessary background materials on ID processes in Section \ref{sec:backgr} and give examples of ID processes in Section \ref{sec:examples:IDP}. In Section \ref{sec:mainresults} we describe a class of ID processes driven by a null recurrent Markov chain and state large and moderate deviation principle for the partial sums of these processes. Section \ref{sec:proofs} proves the results. Lastly, in Section \ref{sec:exam} we discuss some examples of this class of processes.

\end{section}

\begin{section}{Notations and Background Materials} \label{sec:backgr}
\subsection{The L\'evy-Khintchine Representation}
The L\'evy-Khintchine representation is a vital tool for the study of ID distributions and processes; see  Theorem 8.1 and the following remarks   in \cite{sato:1999} for  details.
\begin{theorem}\label{thm:levykhintchine } 
	For any $ \mathbb{R}^{ d}$-valued ID random variable $ X$ there exists a unique triplet  $( \Sigma, \nu,\upsilon  )$ such that 
	\begin{equation}\label{eq:charfn} 
 	E \left( e^{i \lambda \cdot X} \right)= \exp  \left\{ - \frac{1}{2}\lambda\cdot \Sigma \lambda  +i \lambda\cdot \upsilon  + \int_{\mathbb{R}^{d}} \left( e^{i \lambda\cdot z}-1- i \lambda\cdot \llbracket z \rrbracket\right) \nu(dz) \right\}   .
\end{equation}
Here $ \Sigma $ is a $ d\times d$ non-negative definite matrix, $ \upsilon \in \mathbb{R}^{ d}$ and $ \nu$ is a measure on $(\mathbb{R}^{d},\mathcal{B}( \mathbb{R}^{d} ))$ satisfying
\[ 
 	\nu(\{   0  \})=0 \ \ \mbox{ and } \ \    \int_{\mathbb{R}^{d}}\left| \llbracket z \rrbracket\right|^{2} \nu( dz )< \infty,
\]
where 
\[ 
 	\llbracket z \rrbracket   :=  \frac{z}{|z| \vee 1}= \left\{  \begin{array}{cl}  z  &\mbox{ if } |z|\le 1 \\ z/|z| & \mbox{ if } |z| >1 \end{array} \right.
\]
\end{theorem}
  The triplet $( \Sigma, \nu,\upsilon  )$ uniquely determines the distribution of $ X $ and is called the generating triplet of $ X $. From the expression of the characteristic function of $ X$ in \eqref{eq:charfn} it is evident that we can write $ X\stackrel{ d}{ =}X^{ (1)}+X^{ (2)}$, where $ X^{ (1)}$ has a Gaussian distribution with mean $ \upsilon $ and covariance matrix  $ \Sigma $ and $ X^{ (2)}$ is independent of $ X^{ (1)}$. Furthermore, this decomposition of $ X$ is unique. $ X^{ (1)}$ and $ X^{ (2)}$ are called the Gaussian and Poissonian component of $ X$, respectively.  The measure $ \nu$ is called the \emph{L\'evy measure} of $ X$, and it  determines the Poissonian component. We will use the notation $ X \sim ( \Sigma ,\nu,\upsilon ) $ to signify that $ X$ is an ID random variable with generating triplet  $( \Sigma, \nu,\upsilon  )$. 
  
 If $ X$ satisfies 
\[ 
 	E[ \exp ( \lambda \cdot X)]< \infty ,\ \ \ \ \mbox{ for all }\lambda \in \mathbb{R}^{ },   
\]  
 then by Theorem 25.17 in \cite{sato:1999} we can get a representation for the  moment generating function of $ X $:
\begin{equation}\label{eq:mgf}
	 E[ \exp ( \lambda \cdot X )]= \exp  \left\{  \frac{1}{2}\lambda\cdot \Sigma \lambda  + \lambda\cdot \upsilon  + \int_{\mathbb{R}^{d}} \left( e^{ \lambda\cdot z}-1-  \lambda\cdot \llbracket z \rrbracket\right) \nu(dz) \right\}.
\end{equation}
This result is vital for the study of large deviations of ID processes.

\subsection{Infinitely Divisible Process}
A process $\mathbb{X}= (X_{ t}:t\in \mathbb{Z})$ is said to be ID if for every $ k\ge 1$ and $ - \infty<n_{ 1}<\cdots<n_{ k}< \infty$, $ (X_{ n_{ 1}},\ldots,X_{ n_{ k}})$ is ID. This is arguably a large class of stochastic processes since the class of ID distributions is a large one. It is possible to prove a characterization of  ID processes  similar to the L\'evy-Khintchine representation; see \cite{maruyama:1970} and \cite{rosinski:2007}. 
\begin{theorem}\label{thm:idp:levykhint} 
Suppose $ \mathbb{X}=(X_{ n})$ is an ID process. Then there exists a unique triplet $ (\Sigma ,\nu,\upsilon )$ satisfying
\begin{enumerate}[(i)] 
        \item  $ \Sigma :\mathbb{Z}\times\mathbb{Z}\to \mathbb{R}$ is a symmetric non-negative definite function, i.e., for every $ k\ge 1$, $ - \infty<n_{ 1}<\cdots<n_{ k}< \infty $ and $ a_{ 1},\ldots,a_{ k}\in \mathbb{R}$,
        \[ 
 	\Sigma(n_{ 1},n_{ 2})=\Sigma (n_{ 2},n_{ 1})   \ \ \mbox{ and } \ \ \sum_{ i,j=1}^{ k} a_{ i}a_{ j} \Sigma (n_{ i},n_{ j})\ge 0,
\]
\item  $ \nu$ is a measure on $(\mathbb{R}^{ \mathbb{Z}},\mathcal{B} (\mathbb{R}^{ \mathbb{Z}} )) $ satisfying
\[ 
 	\nu\big(\{   x:x_{ n}=0  \}\big)=0 \ \ \ \ \mbox{and} \ \ \ \ \int_{ \mathbb{R}^{ \mathbb{Z}}}    \llbracket x_{ n} \rrbracket^{ 2} \nu(dx) < \infty \ \ \ \ \mbox{ for all } n\in \mathbb{Z},
\]
where for any $ x\in \mathbb{R}^{ \mathbb{Z}}$, $ x_{ n}$ is the projection of $ x$ to the $ n$-th coordinate.
\item $ \upsilon $ is an element of $ \mathbb{R}^{ \mathbb{Z}}$
\end{enumerate}
such that for any $ T=\{   - \infty<n_{ 1}<\cdots<n_{ k}< \infty  \}$, 
\[ (X_{ n_{ 1}},\ldots,X_{ n_{ k}})\sim (\Sigma _{ T},\nu\circ p_{ T}^{ -1},p_{ T}(\upsilon )).\]
Here $ \Sigma _{ T}=(\Sigma (i,j))_{ i,j\in T}$ and $ p_{ T}:\mathbb{R}^{ \mathbb{Z}}\to \mathbb{R}^{ k}$ is the projection mapping $ x$ to $ (x_{ n_{ 1}},\ldots,x_{ n_{k}})$.
\end{theorem}

From this representation it is evident that for any ID process $ \mathbb{X}$ there exists independent processes $ \mathbb{X}^{ (1)}$ and $ \mathbb{X}^{ (2)}$ such that $ \mathbb{X}^{ (1)}$ is Gaussian, $ \mathbb{X}^{ (2)}$ is determined by $ \nu$ and $ \mathbb{X}=\mathbb{X}^{ (1)}+\mathbb{X}^{ (2)}$. $ \mathbb{X}^{ (2)}$ is the Poissonian part of $ \mathbb{X}$ and $ \nu$ is the path-space L\'evy measure of $ \mathbb{X}$.

\subsection{Infinitely Divisible Random Measure}
We introduce the notion of an \emph{infinitely divisible random measure} or \emph{IDRM}. \cite{rosinski:2007} provides a comprehensive treatment of this topic. Suppose $ ( \Omega,\mathcal{B},P)$ is a probability space. Let $ S$ be a set and  and $ \mathcal{S}_{ 0}$ be a $ \sigma $-ring of measurable subsets of $ S$.
\begin{definition}\label{def:idrm}
 $(   M( A)  : A\in \mathcal{S}_{ 0}) $ is an infinitely divisible random measure on $ ( S,\mathcal{S}_{ 0})$ with control measure $ m$ if 
\begin{enumerate}[(i)] 
        \item  $ M(\emptyset)=0, P-a.s.$
        \item For every $ (   A_{ n} :n\ge 1 )\subset \mathcal{S}_{ 0}$ pairwise disjoint, $ (M( A_{ n}):n\ge 1  )$ forms a sequence of independent random variables and if  $ \cup_{ i}A_{ i}\in \mathcal{S}_{ 0}$ then
        \[ 
 	M\Big( \bigcup_{ i=1}^{ \infty }A_{ i}\Big)= \sum_{ i=1}^{ \infty }M( A_{ i}), \ \ \ \ P -a.s.
\]
	\item For every $ A\in \mathcal{S}_{ 0}$, $ M( A)$ has an infinitely divisible distribution.
	\item For a set $ A\in \mathcal{S}_{ 0}$, $ m( A)=0$ if and only if $ M( A^{ \prime })=0, P-a.s.$ for every $ A^{ \prime }\subset A, A^{ \prime }\in \mathcal{S}_{ 0}.$
\end{enumerate}
\end{definition}
The most well-known example of an IDRM is the \emph{Poisson Random Measure} or the PRM, see \cite{kallenberg1983random}. An IDRM $ M$ is said to be a PRM($ m$) on $ S$ if $ M$ is an IDRM with control measure $ m$ and for any measurable $ A$ satisfying $ m(A)< \infty$ we have $ M(A)\sim Poi(m(A))$. 

The following theorem characterizes the generating triplet of an IDRM. 
\begin{theorem}\label{thm:gentripidrm}
\begin{enumerate}[(a)] 
        \item  Let $ M$ be an IDRM on $ ( S, \mathcal{S}_{ 0})$ such that for every $ A\in \mathcal{S}_{ 0}$,
        \begin{equation}\label{eq:gentrip}
	M( A)\sim \big( \Sigma ( A), \nu( A), \upsilon ( A)\big).
\end{equation}
Then 
\begin{enumerate}[(i)] 
        \item  $ \Sigma:\mathcal{S}_{ 0}\to \mathbb{R}_{ +}$ is a measure.
        \item $ \nu$ is a bi-measure, i.e., $ \forall A\in \mathcal{S}_{ 0},$ $ \nu( A,\cdot)$ is a measure on $ ( \mathbb{R}, \mathcal{B}( \mathbb{R}))$ and $ \forall B\in \mathcal{B}( \mathbb{R}^{ d})$, $ \nu( \cdot, B)$ is a measure on $ ( S,\mathcal{S}_{ 0}).$
        \item $ \upsilon :\mathcal{S}_{ 0}\to \mathbb{R}$ is a signed measure.
\end{enumerate}
\item If $ ( \Sigma,\nu,\upsilon )$ satisfy the conditions given in ($ a$), then there exists a unique (in the sense of finite dimensional distributions) IDRM $ M$ such that \eqref{eq:gentrip} holds.
\item Let $ ( \Sigma,\nu,\upsilon )$ be as in  ($ a$). Define a measure
\begin{equation}
	m( A)= \Sigma ( A) +|\upsilon |( A) +\int_{ \mathbb{R}} \llbracket x \rrbracket^{ 2}\nu( A,dx)\  \  \  \ A\in \mathcal{S}_{ 0},
\end{equation}
where $ |\upsilon |=\upsilon ^{ +}+\upsilon ^{ -}$ is the Jordan decomposition of measure $ \upsilon $ into positive and negative parts. Then $ m( \cdot)$ is a control measure of $ M$.
\end{enumerate} 
\end{theorem}
Since we prefer to work with $ \sigma $-finite measures we assume that there exists an increasing sequence $ (   S_{ n}  )\subset \mathcal{S}_{ 0}$ such that 
\begin{equation}\label{eq:sigfin}
	S=\bigcup_{ n}S_{ n}.
\end{equation}
We  can extend $ \nu$ to a measure (by an abuse of notation we will call this $ \nu$ as well) on $ ( S\times \mathbb{R}, \mathcal{S}\times \mathcal{B}( \mathbb{R}))$ such that 
\begin{equation}
	\nu(A\times B )=\nu( A,B), \ \ \ \ A\in \mathcal{S}, B\in \mathcal{B}( \mathbb{R}),
\end{equation}
where $ \mathcal{S}:= \sigma ( \mathcal{S}_{ 0})$. Similarly, it is also possible to extend the measurs $ \Sigma$ and $ | \upsilon  |$ on $ ( S, \mathcal{S})$.
Since all the measures are $ \sigma $-finite, we can define measurable  functions
\begin{equation}
	\sigma ^{ 2}( s):= \frac{ d\Sigma }{dm }( s)
\end{equation}
\begin{equation}
	\eta( s):= \frac{ d\upsilon }{dm }( s)
\end{equation}
and a measure kernel $ \rho( x,dx)$ on $ ( S,\mathcal{B}( \mathbb{R}))$ such that
\begin{equation}
	\nu( ds,dx)= \rho( s,dx)m( ds).
\end{equation} 
$ ( \sigma ^{ 2},\rho, \eta)$ is called the \emph{local characteristic} of $ M$ with respect to the control measure $ m$. Intuitively, we can think that 
\[ 
 	M( ds)\sim ( \sigma ^{ 2}( s)m( ds),\rho( s,\cdot)m( ds), \eta( s)m( ds)).   
\]
The following theorem makes this statement more precise.
\begin{theorem}\label{thm:localchar} 
 Under the above notation and condition \eqref{eq:sigfin}, $ \big( \sigma ^{ 2}( s),\rho( s,\cdot),\eta( s)\big)$ is a generating triplet of some ID distribution $ \mu( s,\cdot)$ on $ \mathbb{R},$ $ m$-a.e. Moreover,
 \begin{equation}
	\sigma ^{ 2}( s)+\int_{ \mathbb{R}} \llbracket x \rrbracket^{ 2} \rho( s,dx)+ | \upsilon  |( s)=1, \ \ \ \ m- a.e.
\end{equation}
For every $ B\in \mathcal{B}( \mathbb{R})$, $ s\mapsto \mu( s,B)$ is measurable and thus $ \mu$ is a probability kernel on $ S\times \mathcal{B}( \mathbb{R})$. If
\begin{equation}\label{eq:cumulant}
	C( s,\lambda )=- \frac{ 1}{2 }\sigma ^{ 2}( s)\lambda ^{ 2}+i\eta( s)\lambda +\int_{ \mathbb{R}} ( e^{i\lambda x }-1-i\lambda \llbracket x \rrbracket)\rho( s,dx)
\end{equation}
then $ \int_{ \mathbb{R}}e^{ i\lambda x}\mu( s,dx)=\exp C( s,\lambda )$ and 
\begin{equation}
	E\big[  \exp ( i\lambda M( A)) \big]=\exp \int_{ A}C( s,\lambda )m( ds). 
\end{equation}
\end{theorem}

\subsection{Integration with respect to an IDRM}
Suppose $ M$ is an IDRM on $ ( S,\mathcal{S}_{ 0})$ with control measure $ m$ and local characteristics $ ( \sigma ^{ 2},\rho,\eta)$. We will define integration with respect to $ M$ of a deterministic function $ f:S\to \mathbb{R}$.  As is often the case, we begin by defining the integral for a simple function. By a simple function we understand a finite linear combination of sets from $ \mathcal{S}_{ 0}$, 
\[
 f( s)=\sum_{ j=1}^{ n}a_{ j}I_{ A_{ j}}( s),\ \ \ \ A_{ j}\in \mathcal{S}_{ 0}.
\]
For such a function the integral is defined in an obvious way:
\[ 
 	\int_{ S} f( s)M( ds) = \sum_{ j=1}^{ n} a_{ j}M( A_{ j}).  
\]

In order to extend to integral beyond simple functions we need to define a distance, say $ d_{ M}$, such that if $ d_{ M}( f_{ n},f)\to 0$, then $ \int f_{ n}( s)M( ds)$ converges in probability to some random variable $ X.$ Then we define $ \int f( s)M( ds)=X$.

For a random variable $ X$, let $ \|X\|_{ 0}:= E \big[  | X |\wedge 1 \big]. $ Clearly $ \|X_{ n}-X\|\to 0$ if and only if $ X_{ n} \stackrel{ P}{ \to}X.$ Define for a simple function $ f: S \to \mathbb{R},$
\begin{equation}
	\|f\|_{ M} := \sup_{ \phi\in \Delta} \| \int_{ S}\phi f M( ds)\|_{ 0},
\end{equation}
where 
\begin{equation}
	\Delta := \{   \phi:S\to \mathbb{R} \mbox{ such that }|\phi|\le 1 \mbox{ and has finite range}\}.
\end{equation}
Notice that $ s\mapsto \phi( s)f( s)$ is a simple function and by our definition, $ \|f\|_{ M}$ is well-defined. It is easy to verify that for any simple functions $ f$ and $ g$,
\begin{enumerate}[(i)] 
        \item $  \| f \|_{ M}=0 \Longleftrightarrow f=0 \ \ m-a.e.$
        \item $ \| f+g \|_{ M}\le \| f \|_{ M}+\| g \|_{ M}.$
        \item $ \| \theta f \|_{ M}\le \| f \|_{ M}$, for any $ | \theta  |\le 1.$
\end{enumerate}
These are properties of an $ F$-norm on a vector space. Naturally, $ d_{ M}( f,g):=\| f-g \|_{ M}$ is a metric on the vector space of simple functions.
\begin{definition}\label{defn:integrable} 
We say that a function $ f:S\to \mathbb{R}$ is $ M$-integrable if there exists a sequence $ \{f_{ n}\}$ of simple functions such that
\begin{enumerate}[(a)] 
	\item $ f_{ n}\to f \ \ m$-a.e.
       \item $\lim_{ k,n\to \infty } \| f_{ n}-f_{ k} \|_{ M}=0.$
\end{enumerate}
If (a)-(b) hold, then we define 
\begin{equation}\label{eq:intdefn}
 	\int_{ S}f( s)M( ds)=\lim_{ n\to \infty } \int_{ S}f_{ n}( s)M( ds),   
\end{equation}
where the limit is taken in probability.
\end{definition}

We now state a necessary and sufficient condition for the existence of $ \int f dM.$ 
\begin{theorem}\label{thm:integrable} 
 A measurable function $ f:S\to \mathbb{R}$ is $ M$-integrable if and only if 
 \begin{equation}
	\int_{ S}\Phi_{ M}( s,f( s))m( ds)< \infty, 
\end{equation}
where
\[
	\Phi_{ M}( s,x)=\sigma ^{ 2}( s)x^{ 2}+\int_{ \mathbb{R}} \llbracket xy \rrbracket^{ 2}\rho( s,dy)+\Big| \eta( s)x+\int_{ \mathbb{R}}\big( \llbracket xy \rrbracket-x \llbracket y \rrbracket \big)\rho( s,dy)  \Big|
\]

If $ f$ is $ M$-integrable, then the integral $ \int f dM$ is well-defined by \eqref{eq:intdefn}, i.e., it does not depend on a choice of a sequence $ \{   f_{ n}  \}$. The integral has an infinitely divisible distribution with characteristic function 
\begin{equation}
	E  \Big[\exp\Big(  i \lambda \int_{ S} f dM\Big) \Big]= \exp  \left\{ \int_{ S}C( s,\lambda f( s))m( ds) \right\},
\end{equation}
where the function $ C$ is as defined in \eqref{eq:cumulant}.
\end{theorem}

\subsection{Miscellany}
We state a few definitions and notation that we will use  in this paper.
\begin{enumerate}[(1)] 
        \item  A function $ f:\mathbb{R}_{ +}\to \mathbb{R}_{ +}$ is said to be \emph{regularly varying} of index $ \alpha $ or $ RV_{ \alpha }$ if for any \[ 
 	\lim_{ t \rightarrow \infty } \frac{ f(tx)}{f(t) } =x^{ \alpha } \ \ \ \ \mbox{ for all }x>0.   
\]
\item For any $ f\in RV_{ \alpha }$ with $ \alpha >0$ define the inverse function of $ f$ as 
\[ 
 	f^{ \leftarrow }(x):= \inf \{   y\ge 1: f(y)\ge x  \} \in RV_{ 1/\alpha }.
\]
\item $ \mathcal{D}$ will denote the space of all function on $ [ 0,1]$ which are right continuous with left limits.  We will use subscripts to denote the topology on the  space. Specifically, the subscripts $S$, $ Sk$ and $P$  will denote the sup-norm topology, the Skorohod topology and the topology of pointwise convergence.
\item For any $ x\in \mathbb{R}^{ \mathbb{Z}}$ or $ x\in \mathbb{R}^{ \mathbb{N}}$, $ x_{ n}$ will denote the projection of $ x$ to the $ n$-th coordinate. 
\item We will  denote by $ L:\mathbb{R}^{ \mathbb{Z}}\to \mathbb{R}^{ \mathbb{Z}}$ the left shift operator defined by
\[ 
 	L\big( x \big)_{ n}=x_{ n+1} \ \ \ \ \mbox{ for all } x\in \mathbb{R}^{ \mathbb{Z}}.   
\]
\item $ \delta_{ x}(\cdot) $ denotes the Dirac delta measure which puts unit mass at $ x$.
\end{enumerate}
\end{section}

\begin{section}{Examples of Infinitely Divisible Processes}\label{sec:examples:IDP} 
In this section we discuss a range of examples of ID processes and describe their L\'evy measures. We exclude  Gaussian processes from our discussion since they have been extensively studied. All infinitely divisible random variables, processes or random measures that we consider henceforth will be Poissonian and therefore without a Gaussian component.

\subsection{Sequence of IID Infinitely Divisible Random Variables}
If $ (X_{ 1},X_{ 2})$ is ID then $ X_{ 1}$ and $ X_{ 2}$ are also ID. Suppose $ \nu,\nu_{ 1}$ and $ \nu_{ 2}$ are the L\'evy measures of $ (X_{ 1},X_{ 2}),X_{ 1}$ and $ X_{ 2}$, respectively. It is then an easy exercise to check that $ X_{ 1}$ and $ X_{ 2}$ are independent if and only if 
\[ 
 	\nu(A)= \nu_{ 1}\big( \{   x: (x,0)\in A \} \big)   +\nu_{ 2}\big( \{   x: (0,x)\in A \} \big) \ \ \ \ \mbox{ for all } A\in \mathcal{B}\big(\mathbb{R}^{ 2}\big),
\] 
i.e.,  the L\'evy measure of $ (X_{ 1},X_{ 2})$ is supported on the axes of $ \mathbb{R}^{ 2}$. If $ X_{ 1}$ and $ X_{ 2}$ are identically distributed then obviously we will have $ \nu_{ 1}=\nu_{ 2}$. Extending this  we get that if $ \mathbb{X}=(X_{ n}:n\in \mathbb{Z})$ is a sequence of iid infinitely divisible random variables with L\'evy measure $ \nu_{ 1}$ then the process $ \mathbb{X}$ has L\'evy measure $ \nu$ given by
\[ 
 	\nu(A)=\sum_{ i\in \mathbb{Z}} \nu_{ 1}\big( \{   x: xL^{ i}(\mathbf{1}) \in A \} \big)    \ \ \ \ \mbox{ for all }A\in \mathcal{B}\big(\mathbb{R}^{ \mathbb{Z}}\big),
\]
where $ \mathbf{1}\in \mathbb{R}^{ \mathbb{Z}}$ is an element such that $ \mathbf{1}_{ 0}=1$ and $ \mathbf{1}_{ i}=0$ for every $ i\ne 0$.

\subsection{Independent ID Random Variables}\label{subsec:ididp}
If $\mathbb{X}= (X_{ n}:n\in\mathbb{Z} )$ is a sequence of independent infinitely distributed random variables then  the L\'evy measure $ \nu$ of the process $\mathbb{X} $ satisfies
\[ 
 	\nu\big( \{   x\in \mathbb{R}^{ \mathbb{Z}}: x_{ i}\neq 0, x_{ j}\ne 0 \mbox{ for } i\ne j  \} \big)    =0,
\]
i.e., the measure $ \nu$ is supported on sequences $ x\in \mathbb{R}^{ \mathbb{Z}}$ such that at most one coordinate is non-zero.

\subsection{L\'evy Processes}
The most well-known examples of ID processes are the L\'evy processes. A L\'evy process is a stochastic process with stationary and independent increments. A L\'evy process indexed by $ \mathbb{N}:=\{   1,2,3,\ldots  \}$ is the partial sums process of a sequence iid infinitely divisible random variables. It is characterized by its L\'evy measure $ \nu$ which is of the form: 
\[ 
 	\nu( A )     =\sum_{ n=1}^{ \infty}\nu_{ 1}\big( \{   x:x\mathbf{1^{ n}}\in A  \}\big)  \ \ \ \ \mbox{ for all }A \in \mathcal{B}\big(\mathbb{R}^{ \mathbb{N}}\big),
\]
where $ \nu_{ 1}$ is the L\'evy measure of $ X_{ 1}$ and for every $ n\ge 1$, $\mathbf{1^{ n}}\in \mathbb{R}^{ \mathbb{N}} $ is such that 
\[ 
 	(\mathbf{1^{ n}})_{ i}=  \left\{\begin{array}{ll}  1  & \mbox{if }i\le n \\0 & \mbox{otherwise}. \end{array}   \right. 
\]

\subsection{Moving Average Processes}
Suppose $ (   Z_{ n}:n\in \mathbb{Z}  )$ is a sequence of iid infinitely divisible random variables having zero mean, finite variance and L\'evy measure $ \nu_{ 1}$. Let $ \phi=(\phi_{ i}:i\in \mathbb{Z})$ be a doubly infinite sequence satisfying $ \sum_{ i\in \mathbb{Z}}\phi_{ i}^{ 2}< \infty.$  $ \mathbb{X}=(X_{ n})$ is said to be a moving average process with innovations $ (Z_{ n})$ and coefficients $ \phi$ if 
\begin{equation}\label{eq:map} 
 	X_{ n}= \sum_{ i\in \mathbb{Z}} \phi_{ i}Z_{ n-i} \ \ \ \ \mbox{ for all } n\in \mathbb{Z}.   
\end{equation}
A discrete Ornstein-Uhlenbeck process driven by a L\'evy process is an example of a moving average process with ID innovations. Such a process $ \mathbb{X}$ has L\'evy measure $ \nu$ given by
\[ 
 	\nu(A)= \sum_{ i\in \mathbb{Z}} \nu_{ 1} \big( \{   x: xL^{ i}(\phi)\in A  \} \big)    \ \ \ \ \mbox{ for all }A \in \mathcal{B}\big(\mathbb{R}^{ \mathbb{Z}}\big).
\]
It is also easy to describe the moving average process in \eqref{eq:map} as a process obtained from an IDRM. Suppose $ M$ is an IDRM on $ \mathbb{R}^{ \mathbb{Z}}$ with local characteristics $ (0,\nu_{ 1},0)$ and control measure $ m$ defined by
\[ 
 	m=\sum_{ i\in \mathbb{Z}} \delta _{ L^{ i}(\phi)}  , 
\]
where $ \delta $ is the Dirac delta measure.
If $ (X_{ n}^{ \prime })$ is a process defined by
\[ 
 	X_{ n}^{ \prime }=\int_{\mathbb{R}^{ \mathbb{Z}} }   s_{ -n}M(ds) \ \ \ \ \mbox{ for all } n\in \mathbb{Z},
\]
then $ \mathbb{X}\stackrel{ d}{= }(X_{ n}^{ \prime }) $ in the sense of all finite dimensional distributions. This is because
\[ 
 	   X_{ n}^{ \prime }=\int_{\mathbb{R}^{ \mathbb{Z}} }   s_{- n}M(ds)= \sum_{ i\in \mathbb{Z}} \big( L^{ i}(\phi) \big)_{- n}M\big( L^{ i}(\phi) \big)  =\sum_{ i\in \mathbb{Z}} \phi_{ i} M\big( L^{n- i}(\phi) \big) 
\]
and $ \big( M(L^{ n}(\phi)):n\in \mathbb{Z} \big) $ is an iid sequence of infinitely divisible random variables with L\'evy measure $ \nu_{ 1}$.

\subsection{Stationary ID Processes}
If $ (X_{ n})$ is  a stationary process then its L\'evy measure on $ \mathbb{R}^{ \mathbb{Z}}$ is shift invariant:
\[ 
 	\nu\circ L^{ -1}=\nu.   
\]

\subsection{A System of Particles} \label{particles}
Consider a system of particles residing on $ \mathbb{Z}$ such that $ Z_{ i,t}$ denotes the position of particle $ i$ at time $ t$ for  $ i\in \mathbb{N}$ and $ t\in \mathbb{Z}_{ +}:=\mathbb{N}\cup\{   0  \}$. Assume that each particle moves in time independent of one another according to a Markov kernel $ p(\cdot,\cdot)$. Furthermore, suppose $ p$ is the transition kernel of a null-recurrent Markov chain with invariant measure $ \pi$ on $ \mathbb{Z}$. Then PRM($ \pi$) is a stationary distribution for $ \sum_{ i=1}^{ \infty}\delta _{ Z_{ i,t}}$; see \citep[p.404-407]{doob1953stochastic} and \cite{spitzer1977stochastic}.  Therefore  assume that  $ \sum_{ i=1}^{ \infty }\delta _{Z_{ i,0}}$ is a PRM($ \pi$) on $ \mathbb{Z}$.  Denote by $ Z_{ i}=(Z_{ i,t}:t\ge 0)$ the path of the $ i$-th particle. Then $M= \sum_{ i=1}^{ \infty }\delta _{ Z_{ i}}$ is a PRM($ m$) on $ (\mathbb{Z})^{ \mathbb{Z}_{ +}}$ where 
\begin{align}
	&m(s:(s_{ 0},\ldots, s_{ k} )\in A_{ 0}\times \cdots \times A_{ k})=\sum\limits_{ s_{ 0}\in A_{ 0}}\cdots \sum\limits_{ s_{ k}\in A_{ k}}  \pi(s_{ 0}) p(s_{ 0},s_{ 1})\cdots p(s_{k-1},s_{ k})\nonumber \\
	 &\mbox{ for all } k\ge 0,\mbox{  and }A_{ 0},\cdots,A_{ k}\subset \mathbb{Z}.
\end{align}

Suppose $ X_{ t}$ is the number of particles residing at $ 0$ at time $ t$. It is easy to observe that $ X_{ t}=\int_{ (\mathbb{Z}_{ +})^{ \mathbb{N}} } I_{ [s_{ t}=0]}M(ds)$ and  $ X_{ t}\sim Poi(\pi(\{   0  \}))$. Furthermore, $ (X_{ t}:t\ge 0)$ is a stationary ID process with  L\'evy measure $ \nu$ given by
\[ 
 	\nu=\sum_{ T\subset\mathbb{Z}_{ +},T\ne \emptyset} m\big( \{   s:s_{ i}=0, \forall i\in T  \} \big) \delta_{ \mathbf{1^{ T}}},
\]
where for every $ \emptyset\ne T\subset \mathbb{Z}_{ +}$, $ \mathbf{1^{ T}}\in \mathbb{Z}_{ +}$ such that
\[ 
 	   (\mathbf{1^{ T}})_{ i}= \left\{   \begin{array}{ll}  1  &\mbox{if }i\in T \\ 0 & \mbox{otherwise}. \end{array}  \right.
\]

\end{section}

\begin{section}{Long Range Dependent ID Processes} \label{sec:mainresults}
In Subsection \ref{subsec:ididp} we discussed the structure of the path space L\'evy measure of a sequence of independent ID random variables. Continuing on  similar lines it is easy to check that an ID process $ (X_{ n})$ is $ m$-dependent if and only if its L\'evy measure $ \nu$ satisfies
\[ 
 	support(\nu)\subset   \big\{ s\in \mathbb{R}^{ \mathbb{Z}}: s_{ i} s_{ j}=0 \mbox{ whenever }| i-j |>m+1   \big\} 
\]
This means that $ \nu$ must supported on sequences for which at most $ m+1$ consecutive coordinates are nonzero and every other coordinate is zero. Now if the measure $ \nu$ is such that for some $ \epsilon >0$
\begin{equation}\label{eq:pathlevy:lrd} 
 	\nu\big( \{   s:|s_{ i}|>\epsilon  \mbox{ for infinitely many } i\in \mathbb{Z}  \} \big)>0    
\end{equation}
then the process $ (X_{ n})$ is long range dependent. We consider a  class of stationary long range dependent infinitely divisible processes by modeling the path space L\'evy measure $ \nu$ so that it satisfies \eqref{eq:pathlevy:lrd}.

%\pagebreak

\subsection{The Set up}
\subsubsection{The space $ S$}
Let $(E,\mathcal{E})$ be a measurable space. Suppose $(Z_n)$ is an irreducible  Harris null-recurrent Markov chain  on $(E,\mathcal{E})$ with transition probabilities $ P( x,\cdot)$ and a $ \sigma$-finite invariant measure $ \pi$.  Define a set $S:=E^{\mathbb{Z}}$ and let $\mathcal{S}$ be the cylindrical $ \sigma -$field on $S$. We define a shift invariant measure $ m$ on $(S,\mathcal{S})$ by
\begin{align}\label{eq:contrmeas}
	&m(s:(s_{ n},\ldots, s_{n+ k} )\in A_{ 0}\times \cdots \times A_{ k})=\int\limits_{ A_{ 0}}\cdots \int\limits_{ A_{ k}}  \pi(ds_{ 0}) P(s_{ 0},ds_{ 1})\cdots P(s_{k-1},ds_{ k})&\nonumber \\
	 &\mbox{ for all } n\in \mathbb{Z}, k\ge 0,\mbox{  and }A_{ 0},\cdots,A_{ k}\in \mathcal{E}.&
\end{align}
We  make certain assumptions on $ ( Z_{ n})$:
\begin{enumerate}[S1.] 
        \item  Assume without loss of any generality (see Remark \ref{rem:mchain}) that the Markov chain has an atom $ a$, i.e.,  $ a \in \mathcal{E}$ is such that
\begin{equation}\label{eq:atom}
 	 \pi(a)>0 \ \ \ \mbox{and} \ \ \  P ( x,\cdot) =  P( y,\cdot) =:P_{ a}(\cdot), \ \ \ \ \mbox{ for all } x,y\in a.  
\end{equation}

\item Define
\[
	T_0=0 \mbox{ and }  T_k:=\inf \{n>T_{k-1}: Z_{ n}\in a  \} , 
\]
which is the time taken by $( Z_n)$ to hit $a$ for the $k$-th time. Sometimes we will also use $T$ to denote $T_{1}$. We  assume that  $ ( Z_{ n})$ is $ \alpha $-regular, that is, 
\begin{equation}\label{eq:gamn} 
 	\gamma ( x):=\big( \pi( a)P_{a}[ T>x ]\big)^{ -1}\in RV_{ \alpha }    
\end{equation} 
for some $ 0< \alpha < 1$. As explained in the proof of Theorem 2.3 and equation (5.17) in \cite{chen:1999}  
\begin{equation}
	\frac{ 1}{\pi( a) }\sum_{ k=1}^{ n}  P^{  k}( a,a) \sim \gamma ( n).
\end{equation}
\item We assume that there exists a measurable partition $ (   E_{ n},n\ge 0  )$ of $ E$ such that:
\begin{enumerate}[(a)] 
\item  $ \pi( E_{ n})<\infty$ for every $ n\ge 0$.
\item There exists a monotone function  $\psi\in RV_{\beta}$ with $ \beta >0$ such that 
\begin{equation}\label{eq:psi}
	Q_{n}(\cdot):=P_{\pi_{ n}}\left[ \frac{T}{ \psi(n)}\in \cdot \right] \Longrightarrow Q( \cdot),
\end{equation}
where $ \pi_{ n}( \cdot)=\pi( \cdot\cap E_{ n})/\pi( E_{ n})$. 
\item There exists $ \zeta>-1$ such that
\begin{equation}\label{eq:pi:regvar}
 	\frac{ \pi( E_{ [ rn]})}{\pi( E_{ n}) } \to r^{ \zeta}, \mbox{ for all } r>0. 
\end{equation}
\item There exists  $ \epsilon^{ \prime } >0$, $ c>0$, $ N>1$, $ k>0$ and $ \epsilon>0$ such that for every $ r\ge 1$
\begin{equation}\label{eq:Qassmp} 
 	 \sup_{ n\ge Nr}P_{ \pi_{ n}} \left[ \frac{ T}{\psi(n ) }\le c r^{ -\beta +\epsilon^{ \prime } } \right]< k r^{ -\zeta -1-\epsilon }.   
\end{equation}
\end{enumerate}

\end{enumerate}

\begin{remark}\label{rem:mchain} 
The assumption that the Markov kernel $ P(\cdot,\cdot)$ admits an atom as defined in \eqref{eq:atom} can be made without loss of any generality. Theorem 2.1 in \cite{nummelin:1984} states that if  the  Markov kernel $ P( x,\cdot)$ is Harris recurrent then it satisfies a minorization condition, i.e, there exists a set $ C\in \mathcal{ E}$,  satisfying $ 0<\pi( C)<\infty,$   a probability measure $ \nu$ on $ ( E,\mathcal{E})$ with $ \nu( C)>0$, such that for some $ 0<b\le 1$ and $ n_{ 0}\ge 1,$
\begin{equation}\label{eq:minor}
	P^{ n_{ 0}}( x,A)\ge b I_{ C}( x)\nu( A), \ \ \ \ \mbox{ for all } x\in E, A\in \mathcal{E}.
\end{equation}
By the split-chain technique it is possible to  embed $ ( Z_{ n})$ into a larger probability space on which one can define a sequence of Bernoulli random variables $ ( \tilde Z_{ n})$, such that $ ( Z^{ \prime }_{ n})=( Z_{ n},\tilde Z_{ n})$ forms a Harris recurrent Markov chain on $E^{ \prime}:= E\times \{   0,1  \}$ with an atom $ a^{ \prime }=E\times \{   1  \}$;  see \cite{athreya:ney:1978}, \cite{nummelin:1978} and section 4.4 in \cite{nummelin:1984}.  Furthermore, there is an  invariant measure $  \pi^{ \prime }$ of $ (  Z^{ \prime }_{ n})$ such that the marginal of $  \pi^{ \prime }$ on $ E$ is $ \pi$ and 
\[
	 \pi^{ \prime }( a)=b \pi( C), 
\]
We also have 
\[ 
 	 P^{\prime  k} ( a,a)=b\nu P^{ k-1}( C), \ \ \ \ k\ge 1,   
\]
where $  P^{ \prime k}$ is the $ k$-step transition of $ (  Z^{ \prime }_{ n})$. This in turn implies
\[
	\frac{ 1}{\pi^{ \prime }( a) }\sum_{ k=1}^{ n}  P^{ \prime k}( a,a)= \frac{ 1}{ \pi( C) } \sum_{ k=1}^{ n}  \nu P^{ k-1}(C) \sim \gamma ( n).
\]
\end{remark}

\subsubsection{The IDRM $ M$}

Suppose that $ \mathcal{S}_{ 0}=\{  A\in \mathcal{S}:m( A)<\infty   \}$ and that  $( M(A):A\in \mathcal{S}_{ 0} )$ is an infinitely divisible random measure on $( S,\mathcal{S}_{ 0})$ with control measure $m$ and local characteristics $(0,\rho,0)$. That means for any $ A\in \mathcal{S}_{ 0}$
\begin{equation}
	E\big[ \exp( i \lambda M( A))\big]= \exp\left \{   \int_{ A} \int_{ \mathbb{R}} \left( e^{ i\lambda z}-1-i\lambda \llbracket z \rrbracket \right)\rho( s,dz)m( ds)\right \}.
\end{equation} 
\begin{enumerate}[M1.] 
        \item  We assume that  $ \rho(s,\cdot)$ is a L\'evy measure  on $ ( \mathbb{R}, \mathcal{B}( \mathbb{R}))$, which is independent of $ s\in S$, i.e.,
\[ 
	 \rho( s,\cdot)=\rho( \cdot), \ \ \ \ m-a.s.
\] 
\item We also assume that the L\'evy measure satisfies the condition:
\begin{equation}\label{eq:mgflocal}
	\int_{ |z|>1}e^{ \lambda z} \rho( dz)< \infty,\ \ \ \ \mbox{ for all } \lambda \in \mathbb{R}.
\end{equation}
Then by \eqref{eq:mgf} we get  $ E[  \exp( \lambda M( A))]< \infty $ for all $ \lambda \in \mathbb{R}$ and 
\begin{equation}
	E \big[  \exp( \lambda M( A)) \big]= \exp  \left\{ \int_{ A}\int_{ \mathbb{R}} \left( e^{ \lambda z}-1 - \lambda \llbracket z \rrbracket\right)\rho( dz)m( ds) \right\}
\end{equation}
We  use $ g:\mathbb{R} \to \mathbb{R}$ to denote the function 
\begin{equation}\label{eq:g}
	g( \lambda ):= \int_{ \mathbb{R}} \left( e^{ \lambda z}-1 - \lambda \llbracket z \rrbracket\right)\rho( dz)
\end{equation}
and that  implies
\[ 
 	   E \left[  \exp( \lambda M( A)) \right]=\exp  \left\{ g( \lambda )m( A) \right\} \ \ \ \ \mbox{ for all } \lambda \in \mathbb{R}.
\]
\end{enumerate}

\subsubsection{The process $ (X_{ n})$}

Finally we define a stationary infinitely divisible process  
\begin{equation}\label{eq:idproc}
	X_n=\int_Sf(s_n)M(ds), n\in \mathbb{Z}
\end{equation} where $f:E  \rightarrow \mathbb{R}$ is a measurable function satisfying:
\begin{enumerate}[F1.]
	\item $ f\in L_{ 1}( E,\mathcal{E},\pi)$ with
	\begin{equation}\label{eq:cf}
	c_{ f}:= \int_{ E} f( x)\pi( dx)\neq 0.
\end{equation}
	\item $ f$ satisfies the conditions of Theorem \ref{thm:integrable}.
	\item For every $ \lambda >0$ there exists $ k>0$, $ N>1$ and $ \epsilon >0$ such that  for every  $ r\ge 1$
\begin{equation}\label{eq:fassmp}
 	\sup_{ n\ge Nr} E_{ \pi_{ n}}  \Bigg[ g \Big(   \lambda\sum_{ i=1}^{ T\wedge \psi( n/r)}\big|f( Z_{ i})\big| \Big) \Bigg]   \le kr^{ -\zeta -1-\epsilon } .
\end{equation}
\item $ X_{ 1}$ has mean 0. We will use $ \sigma ^{ 2}$ to denote the variance of $X_{ 1}$.
\end{enumerate}

\subsection{Series Representation}

A series representation of the process $ (X_{ t})$ in \eqref{eq:idproc} gives a representation in the form described in \eqref{eq:vaguedescofX}.  Series representation of ID random variables and methods of simulation using these representations are well known, see \cite{ferguson1972representation}, \cite{bondesson1982simulation}, \cite{rosinski:1990} and \cite{rosinski:2007}. 

Suppose $ \Gamma =(   \Gamma _{ i}:i\ge 1  )$ are the arrival times of a unit rate Poisson process and $ ( \xi_{ i},i\ge 1   )$ is a sequence of i.i.d. random variables independent of $ \Gamma $. Also, suppose $ H:(0.\infty)\times \mathbb{R}\to \mathbb{R}$ be a measurable function such that 
\begin{equation}\label{eq:localchar} 
 	\int_{ 0}^{ \infty} \left| E \llbracket H(s,\xi_{ 1})\rrbracket \right|ds <\infty   \ \ \ \ \mbox{ with } \ \ \ \ a:=\int_{ 0}^{ \infty}  E \llbracket H(s,\xi_{ 1})\rrbracket ds.
\end{equation}
Following \cite{rosinski:1990} we then know that 
\[ 
 	\sum_{ i=1}^{ \infty} H(\Gamma _{ i},\xi_{ i})\sim (0,\nu_{ H},a)   
\]
where 
\[ 
 	\nu_{ H}(B)=\int_{ 0}^{ \infty}P\big[ H(s,\xi_{ 1})\in B\setminus \{   0  \}\big]    ds.
\]

Next, suppose $ N=\sum_{i\ge 1 }\delta _{ \Gamma _{ i},Z_{ i,0}}\sim PRM(Leb\times \pi)$ on $ (0,\infty)\times E$  and $ ( \xi_{ i},i\ge 1   )$ is a sequence of i.i.d. random variables independent of $N $. If $ H$ satisfies \eqref{eq:localchar}, then $ \tilde M=\sum_{ i\ge 1} H(\Gamma _{ i},\xi_{ i})\delta _{ Z_{ i,0}}$ is an IDRM on $ E$ with control measure $ \pi$ and local characteristics $ (0,\nu_{ H},a)$. Suppose at each location $ Z_{ i,t}$ we start an independent Markov chain with kernel $ P$ then 
\[ 
 	M=\sum_{ i\ge 1} H(\Gamma _{ i},\xi_{ i}) \delta _{ (Z_{ i,0},Z_{ i,1},Z_{ i,2},\ldots)}   
\]
is an IDRM on $ S=E^{ \mathbb{Z}}$ with control measure $ m$ in \eqref{eq:contrmeas} and local characteristics $ (0,\nu_{ H},a)$. Furthermore, for a function $ f:E\to \mathbb{R}$ satisfying conditions F1 and F2 
\[ 
 	X_{ t}=\sum_{ i\ge 1}H(\Gamma _{ i},\xi_{ i})f(Z_{ i,t})   
\]
is an ID process without any Gaussian component and with L\'evy measure $ \nu_{ H}$. 

We can view this representation in two ways. Firstly, if we are given a L\'evy measure $ \nu$ then we can find a suitable function $ H$ such that $ \nu=\nu_{ H}$ and then this representation will be helpful for simulation.  Secondly, if we are interested in a system of particles with weights where the weights can be expressed in the form $ W_{ i}=H(\Gamma_{ i},\xi_{ i})$ then we can study functionals of the form $ \sum_{ i\ge 1}W_{ i}f(Z_{ i,t})$ using the structure of ID processes.

\subsection{Large and Moderate Deviation Principle}

Suppose $ (X_{ t})$ is as defined in \eqref{eq:idproc}. We study the step process  $(Y_n)$ 
\begin{equation}\label{stepjump:id} 
Y_n(t)=\frac{1}{a_n}\sum_{i=1}^{[nt]}X_i,t\in [0,1],
\end{equation} 
and its polygonal path counterpart
\begin{equation}\label{e:polyg:id}
\tilde Y_n(t)=\frac{1}{a_n} \sum_{i=1}^{[nt]}X_i
+\frac{1}{a_n}(nt-[nt])X_{[nt]+1},t\in[0,1], 
\end{equation}
where $ (a_{ n})$ is a suitable normalizing sequence. Let $ \mu_{ n}$ be the law of $ Y_{ n}$ and $ \tilde \mu_{ n}$ be the law of $ \tilde Y_{ n}$, in some appropriate function space equipped with the cylindrical $ \sigma -$field. We use $\mathcal{BV}$ to denote the space of all real valued functions  of bounded variation defined on the unit interval $[0,1]$. $ \mathcal{AC}$ will stand for functions which are absolutely continuous on $ [0,1]$. To ensure that the space $\mathcal{BV}$ and $ \mathcal{AC}$ are  measurable sets in the cylindrical $\sigma$-field of all real-valued functions on $[0,1]$, we use  only rational partitions of $[0,1]$ when defining variation. 

\begin{assumption}\label{assmp:integ:g:f} 
 There exists $ 1\le\delta<( 1-\alpha )^{ -1} $  such that the following hold:
 \begin{enumerate}[(1)]
 \item The function $ g$ defined in \eqref{eq:g} satisfies the integrability condition
 \begin{equation}\label{eq:integ}
	\int_{ 0}^{ \infty } \exp \left( -k_{ 0}\bar{ g}( t)^{ \delta   } \right)dt<\infty,
\end{equation}
for some $ k_{ 0}>0$, where $ \bar{ g}( t)=\min\{   |s|:g( s)=t  \}$ and 
\item The function $ f$ satisfies
\begin{equation}\label{eq:Gamma} 
	\sup_{ x\in E}  E_{ x} \Big[  \exp \Big( \lambda  \sum_{ i=1}^{ T}\big| f( Z_{ i})\big| \Big)^{ \delta } \Big]< \infty 
\end{equation}
for some $ \lambda >0$.
\end{enumerate}
\end{assumption}

\begin{theorem}[Moderate Deviation Principle]\label{thm:moderate}
	Consider the setup described above and assume that Assumption \ref{assmp:integ:g:f} holds. Let $   \mu_{ n}  $ be the law of $ Y_{ n}$ in $ \mathcal{BV}$ where 
\[ 
 		a_{ n}:=\frac{\pi( E_{ [ \psi^{ \leftarrow}( n)]})\gamma ( n)\psi^{ \leftarrow}( n)}{c_{ n}}, 
\]
and $ c_{ n}\to \infty$ is such that $ b_{ n}=\pi( E_{[  \psi^{ \leftarrow}( n)]})\psi^{ \leftarrow}( n)/c_{ n}^{ 2}\to \infty$.  Then $ (\mu_{ n})$ satisfies large deviation principle in $ \mathcal{BV}_{ S}$ with speed $ (b_{ n})$ and good rate function
\begin{equation}\label{eq:ratefn:id:mod}
	H_{ m}( \xi)= \left\{  \begin{array}{cc}  \Lambda^{ *}_{ m}( \xi^{ \prime })   & \mbox{ if } \xi\in \mathcal{AC}, \xi( 0)=0\\ \infty & \mbox{ otherwise.} \end{array} \right.
\end{equation}
where for any $ \varphi\in L_{ 1}[ 0,1]$
\begin{equation}\label{eq:ratefnform:id:mod}
	\Lambda^{ *}_{ m }( \varphi)=\sup_{ \psi\in L_{ \infty }[ 0,1]}\left\{ \int_{ 0}^{ 1} \psi( t)\varphi( t) dt -     \int_{0}^{ \infty }r^{ \zeta }E  \Big[ I_{[  r^{ \beta }V \le 1]} \frac{ \sigma ^{ 2}}{2 } \Big(c_{ f}   \int\limits_{ 0}^{1- r^{ \beta }V } \psi( t) U( dt) \Big)^{ 2} \Big]dr \right\}.
\end{equation}
Here  $    U( t):=\inf\{  x: S_{ \alpha }( x)\ge t  \},0\le t\le 1,  $ is the inverse time $ \alpha -$stable subordinator with
\[ 
 	 E  \left\{ \exp \left(- \lambda  S_{ \alpha }( 1) \right) \right\}= \exp  \left\{ - \frac{ \lambda ^{ \alpha }}{\Gamma( 1+\alpha ) } \right\},\forall \lambda \ge 0,   
\]
and $ V $  is independent of $ \{   U( t):0\le t\le 1  \}$ having distribution $ Q( \cdot)$.
\end{theorem}

\begin{theorem}[Large Deviation Principle]\label{thm:main}
Assume the setup described above and Assumption \ref{assmp:integ:g:f} such that \eqref{eq:Gamma} holds for all $ \lambda \in \mathbb{R}$. Let $   \mu_{ n}  $ be the law of $ Y_{ n}$ in $ \mathcal{BV}$ where 
\[ 
 		a_{ n}:=\pi( E_{ [ \psi^{ \leftarrow}( n)]})\gamma ( n)\psi^{ \leftarrow}( n). 
\]
Then  $ ( \mu_{ n}  )$ satisfies large deviation principle in $ \mathcal{BV}_{ S}$ with speed $ b_{ n}=\pi( E_{[  \psi^{ \leftarrow}( n)]})\psi^{ \leftarrow}( n)$ and good rate function
\begin{equation}\label{eq:ratefn:id}
	H( \xi)= \left\{  \begin{array}{cc}  \Lambda^{ *}( \xi^{ \prime })   & \mbox{ if } \xi\in \mathcal{AC}, \xi( 0)=0\\ \infty & \mbox{ otherwise.} \end{array} \right.
\end{equation}
where for any $ \varphi\in L_{ 1}[ 0,1]$
\begin{equation}\label{eq:ratefnform:id}
	\Lambda^{ *}( \varphi)=\sup_{ \psi\in L_{ \infty }[ 0,1]}\left\{ \int_{ 0}^{ 1} \psi( t)\varphi( t) dt -     \int_{0}^{ \infty }r^{ \zeta }E  \Big[ I_{[  r^{ \beta }V \le 1]}g \Big(c_{ f}   \int\limits_{ 0}^{1- r^{ \beta }V } \psi( t) U( dt) \Big) \Big]dr \right\}
\end{equation}
with  $ U$ and $ V$  as described in Theorem \ref{thm:moderate}.
\end{theorem}

\begin{remark}\label{rem:mainresults} 
 Theorems  \ref{thm:moderate} and  \ref{thm:main}   demonstrate the effect of memory on the large deviation principle of the process. Recall  that for an i.i.d. sequence $ (X_{ n})$ the normalized partial sums $ S_{ n}/n$ satisfy LDP with a linear  speed sequence (Cram\'er's Theorem, see Theorem 2.2.3 in \cite{dembo:zeitouni:1998}). The normalizing sequence and the speed sequence both are linear. This is also the case for L\'evy processes as is proved in \cite{deAcosta:1994}. Theorem \ref{thm:main} shows that in our setup the speed sequence grows at a much slower rate than the normalizing sequence. Here we see that $ (a_{ n})$ is regular varying of index $ (\zeta +1+\alpha \beta )/\beta $, $ (b_{ n})$ is regular varying of index $ (\zeta +1)/\beta $ and $ a_{ n}/b_{ n}=\gamma (n)\to \infty$.  This phenomenon is observed in the moderate deviation regime as well as described in Theorem \ref{thm:moderate}. In the i.i.d. case the normalizing sequence is $ a_{ n}=n/c_{ n}$ and moderate deviation principle holds with speed sequence $ b_{ n}=n/c_{ n}^{ 2}$ where $ c_{ n}\to \infty$ is such that $ n/c_{ n}^{ 2}\to \infty$. A similar phenomenon for moving average processes was demonstrated  in \cite{ghosh:samorodnitsky:2009}. We discuss  two examples in Section \ref{sec:exam}.
\end{remark}

\begin{remark}\label{rem:lrdhistory} 
 Here we see a connection between the recurrence of markov chain $ (Z_{ n})$ and long range dependence of the process $ (X_{ n})$.  The relation is that because of recurrence  the same part of the random measure $ M$ contributes to the variables $ (X_{ n})$ significantly infinitely often.  The parameter $ \alpha $ plays an important role of determining how strong the memory is. The higher the value of $ \alpha $ the sooner are the returns of the Markov chain and that results in a longer  memory in the process $ (X_{ n})$. Furthermore,
  the proofs of the theorems will show that we make a very interesting connection between the weak convergence of functionals  of Markov chain and the large deviations of the process $ (X_{ n})$.
  
  Such models have been considered before but in the context of symmetric $ \alpha $- stable (S$ \alpha $S) processes. A S$ \alpha $S process can be generated from the model we discussed above if the  IDRM $ M$ is a S$ \alpha $S random field.  An IDRM $ M$ is S$ \alpha $S if the local characteristics are $ (0,\rho, 0)$ with $ \rho(dx)=c|x|^{ -\alpha -1}dx$ for some $ c>0$. These processes are not included in the class we consider  since they do not permit exponential moments.  \cite{rosinski1996cms} showed that  the shift operator $ L$  on $ S$ is measure-preserving and conservative and the S$ \alpha $S process thus generated is mixing. \cite{mikosch2000rpc} and later \cite{alparslan2007rpc} discussed the asymptotic properties of ruin properties (as the initial reserve increases to infinity) where the claims process is an S$ \alpha $S process just described. 
\end{remark}

\begin{remark}\label{rem:assmp4-2} 
 We use the Gartner-Ellis type argument to prove LDP for $ Y_{ t}$. Hence a major step in the proof is to obtain a proper estimate of the log-moment generating function of $ Y_{ t}$. As will be evident in Section \ref{sec:proofs}, an important tool in the proof is the idea of split chain technique of dividing a Markov path into i.i.d. segments, see \cite{athreya:ney:1978} and \cite{nummelin:1978}. Assumption \ref{assmp:integ:g:f} ensures that the contribution of each independent segment is not too large. The first such segment though needs a special treatment and conditions \eqref{eq:Qassmp} and \eqref{eq:fassmp} are needed to ensure that the first segment is negligible.
 
  Assumption \ref{assmp:integ:g:f} makes an integrability assumption on the functionals  $ \sum_{ i=1}^{ T}f(Z_{ i})$ and the rate of growth of the function $ g$. The  trade-off is that, the higher the $ \delta $ the more restricted will be the class of functions $ f$ that we can choose. So for distributions which have slowly growing log-moment generating function, i.e., $ \delta $ is small,  we will be able to choose $ f$ from a bigger class of functions. 
\end{remark}

\end{section}

\begin{section}{Proof of Theorems \ref{thm:moderate} and \ref{thm:main}}\label{sec:proofs}

The proof of Theorems \ref{thm:moderate} and \ref{thm:main} is long and is therefore broken up into several steps in form of lemmas below. As the proofs of both the theorems  proceed in a very similar fashion and so we present them together.

\begin{proof}[Proof of Theorems \ref{thm:moderate} and \ref{thm:main}]
 In the arguments used in the proof and the Lemmas henceforth we will take $ a_{ n}=\pi( E_{ [ \psi^{ \leftarrow}( n)]})\gamma ( n)\psi^{ \leftarrow}( n)/c_{ n}$ and  $b_{ n}=\pi( E_{[  \psi^{ \leftarrow}( n)]})\psi^{ \leftarrow}( n)/c_{ n}^{ 2}  $ where in the case of Theorem \ref{thm:main} we will take  $ c_{ n}=1$, and for Theorem \ref{thm:moderate} $ c_{ n}$ will be as described in the statement of that theorem. Let $\mathcal{X}$ be the set of all  $\mathbb{R}^d$-valued functions defined on the unit interval  $[0,1]$ and let $\mathcal{X}^o$ be the subset of $\mathcal{X}$, of  functions which start at the origin. Define $\mathcal{J}$ as the   collection  of  all ordered finite subsets of $(0,1]$ with a partial    order defined by inclusion. For any $j=\{0=t_{ 0}<t_1<\ldots<t_{|j|}\le    1\}$ define the projection $p_j:\mathcal{X}^o\rightarrow    \mathcal{Y}_j$ as  $p_j(\xi)=(\xi(t_1),\ldots,\xi(t_{|j|}))$,      $\xi\in \mathcal{X}^o$. So $\mathcal{Y}_j$ can be identified with the    space $\mathbb{R}^{|j|}$ and the projective limit of    $\mathcal{Y}_j$ over $j\in \mathcal{J}$  can be identified with    $\mathcal{X}^o_{ P}$, that is, $ \mathcal{X}^{ o}$ equipped with the topology of pointwise  convergence. Note that $\mu_n\circ p_j^{-1}$ is the law of     \[Y_n^j=(Y_n(t_1),\ldots,Y_n(t_{|j|})).\] Define the vector $ V_{ n}$ as     
\begin{equation}\label{id:vn}
	V_n:=\big(Y_n(t_1),Y_n(t_2)-Y_n(t_1),\cdots,Y_n(t_{|j|})- Y_n(t_{|j|-1})\big)
\end{equation} 
and observe that for    any $\underline{\lambda}=(\lambda_1,\ldots,\lambda_{|j|})\in    \mathbb{R}^{|j|}$ 
\begin{align}\label{eq:mentioned:earlier}
	 &	\log  E\big(\exp\big[b_n\underline{\lambda}\cdot V_n \big]\big) \nonumber\\
&=\log  E\exp\Big[\frac{b_n}{a_n}\sum_{i=1}^{|j|}\lambda_i  \Big(\sum_{k=[nt_{i-1}]+1}^{[nt_{i}]} X_k \Big)\Big] \nonumber\\
	& =  \int\limits_{ S}\int\limits_{ \mathbb{R}} \left( \exp\Big( \frac{ b_{ n}}{a_{ n} } z \sum_{ i=1}^{ | j |}\lambda _{ i} \sum_{k=[nt_{i-1}]+1}^{[nt_{i}]} f( s_{ k})\Big)-1 - \frac{ b_{ n}}{a_{ n} }\llbracket z \rrbracket \sum_{ i=1}^{ | j |}\lambda _{ i} \sum_{k=[nt_{i-1}]+1}^{[nt_{i}]} f( s_{ k})\right) \rho( dz)m( ds)\nonumber\\
	& =   \int_S   g \left( \frac{ b_{ n}}{a_{ n} }  \sum_{ i=1}^{ | j |}\lambda _{ i} \sum_{k=[nt_{i-1}]+1}^{[nt_{i}]} f( s_{ k}) \right) m(ds)\nonumber.
\end{align}
 By Lemma \ref{lem:id:limit}
\[ 
 	\lim_{ n \rightarrow \infty   } \frac{ 1}{b_{ n} }   \log  E\big(\exp\big[b_n\underline{\lambda}\cdot V_n \big]\big) = \Lambda_{ j}(\underline{\lambda })
\]
with
\[ 
 	\Lambda _{ j}( \underline{ \lambda }):=    \int_{0}^{ \infty }r^{ \zeta }E  \Big[ I_{[  r^{ \beta }V \le 1]}\tilde g \Big(c_{ f}   \sum\limits_{ i=1}^{ | j |}\lambda _{ i} U(  t_{ i}- r^{ \beta}V )-U( t_{ i-1}- r^{ \beta }V )  \Big) \Big]dr,
\]
where $ V $ and $ (U(t ):0\le t\le 1)$ are as described in the statement of the theorems and
\[ 
 	\tilde g(x) =\left\{    \begin{array}{ll} \sigma^{ 2}x^{ 2}/2  & \mbox{for  Theorem \ref{thm:moderate}} \\ g(x) & \mbox{for  Theorem \ref{thm:main}.} \end{array}  \right.   
\] We understand $ U( t)=0$ for all $ t<0$. By the Gartner-Ellis Theorem (see \cite{gartner:1977}, \cite{ellis:1984} and Theorem 2.3.6 in \cite{dembo:zeitouni:1998}) the laws of $(V_n)$ satisfy LDP with speed $(b_n)$ and good  rate function
\begin{equation}\label{eq:fidirtfn}
\Lambda^{*}_{j}(w_1,\ldots,w_{|j|})=\sup_{ \underline{ \lambda }}\Big\{   \sum_{ i=1}^{ | j |}w_{ i}\lambda _{ i}-\Lambda _{ j}( \underline{ \lambda }) \Big \}, 
\end{equation}
where $(w_1,\ldots,w_{|j|})\in \mathbb{R}^{|j|}$. The map $V_n\mapsto Y_n^j$ from $\mathbb{R}^{|j|}$ onto itself is one to one and continuous. Hence the contraction principle tells us that  $(\mu_n\circ p_j^{-1})$ satisfy LDP in $\mathbb{R}^{|j|}$ with   good rate function  
\begin{equation}\label{Hsfindim:id}
H_{j}(y_1,\ldots,y_{|j|}) := \Lambda^{*}_{j}(y_1,y_{ 2}-y_{ 1},\ldots,y_{|j|}-y_{ | j |-1}).
\end{equation}  
By Lemma \ref{lem:expequiv:id}, the same holds for the measures \mbox{$(\tilde \mu_n\circ p_j^{-1})$}. By the Dawson-Gartner Theorem (Theorem 4.6.1  in \cite{dembo:zeitouni:1998}) this implies  that the measures $(\tilde \mu_n)$ satisfy LDP in the space $\mathcal{X}^o$ equipped with the topology of pointwise convergence, with speed $b_n$ and the rate function 
\[ 
 	I(\cdot):=\sup_{ j\in \mathcal{J}} H_{ j}( p_{ j}( \cdot))   
\]
which by Lemma  \ref{lem:ratefn:id} is same as the function $ H_{ m}( \cdot)$ in the case of Theorem \ref{thm:moderate} and $ H(\cdot)$ in the case of Theorem \ref{thm:main}. As $\mathcal{X}^o$ is a closed subset of $\mathcal{X}$, by Lemma 4.1.5 in \cite{dembo:zeitouni:1998} the same holds for $(\tilde\mu_n)$ in $\mathcal{X}$ and the rate function is infinite outside $\mathcal{X}^o$.  Since $\tilde \mu_n(\mathcal{BV})=1$ for all $n\ge 1$ and $ I( \cdot)$ is infinite outside of $\mathcal{BV}$, we conclude that $(\tilde \mu_n)$ satisfy  LDP in $\mathcal{BV}_P$ with the same rate function. The sup-norm topology on $\mathcal{BV}$ is stronger than that of pointwise convergence  and  by Lemma \ref{lem:expttnsC:id}, $(\tilde \mu_n)$ is exponentially tight in $\mathcal{BV}_S$. So by corollary 4.2.6 in 
\cite{dembo:zeitouni:1998}, $(\tilde \mu_n)$ satisfy LDP in $\mathcal{BV}_S$ with speed $b_n$ and good rate function $I(\cdot)$. Finally, applying Lemma \ref{lem:expequiv:id} once again, we conclude that the same is true for the sequence $( \mu_n)$. 
\end{proof}

\begin{lemma}\label{lem:expequiv:id} 
If the conditions of Theorem \ref{thm:moderate} or Theorem   \ref{thm:main} hold then the families $\{\mu_n\}$ and $\{\tilde \mu_n\}$ are exponentially  equivalent in $\mathcal{D}_S$, i.e., for any $ \epsilon >0$
\[ 
 	\lim_{ n \rightarrow \infty } \frac{ 1}{b_{ n} } \log P \Big[  \big\| Y_{ n}-\tilde Y_{ n}\big\| >\epsilon  \Big]   =- \infty.
\]
\end{lemma}
\begin{proof}
Observe that for every $ n\ge 1$
\[ 
 	\big\|Y_{ n}-\tilde Y_{ n}\big\| \le \frac{ 1}{a_{ n} }\max_{ 1\le i\le n}|X_{ i}|.   
\]
Therefore for any $\epsilon >0$ and $\lambda>0$  
\begin{eqnarray*}
&& \limsup_{n\rightarrow \infty} \frac{1}{b_n}  \log P\big( \big\|Y_n-\tilde Y_n\big\|>\epsilon \big) \\
& \le & \limsup_{n\rightarrow \infty} \frac{1}{b_n}\log P\Big( \frac{1}{a_n}\max_{1\le i\le n}|X_i|>\epsilon \Big)\\ 
& \le & \limsup_{n\rightarrow \infty} \frac{1}{b_n} \log \Big( n\, P\big(|X_1|>a_n\epsilon \big) \Big)\\
& \le & \limsup_{n\rightarrow \infty} \frac{1}{b_n} \Big( \log n -a_n\lambda \epsilon  + \log E \big( e^{ \lambda X_{ 1}}\big) + \log E \big( e^{ -\lambda X_{ 1}}\big)\Big)\\
& = & \limsup_{n\rightarrow \infty} \frac{1}{b_n}\Big( -a_n\lambda \delta\Big) \,.
\end{eqnarray*}
By the definition of $ a_{ n}$ and $ b_{ n}$ in Theorem \ref{thm:moderate} or Theorem   \ref{thm:main} we have $a_n/b_n\to\infty$, so the above limit is equal to $-\infty$ and that completes the proof. 
\end{proof}

\begin{lemma}\label{lem:id:limit} 
Assume that the conditions of Theorem \ref{thm:moderate} of Theorem \ref{thm:main} hold. Then for any $ j= \{   0=t_{ 0}<t_{ 1}<\cdots<t_{ | j |}  \}\in \mathcal{J}$ and $ ( \lambda _{ 1},\ldots,\lambda _{ | j |})\in \mathbb{R}^{ | j |}$, 
 \begin{eqnarray*}
	 &  & \lim_{ n \to \infty } \frac{ 1}{ b_{ n}} \int_S   g \left( \frac{ b_{ n}}{a_{ n} }  \sum_{ i=1}^{ | j |}\lambda _{ i} \sum_{k=[nt_{i-1}]+1}^{[nt_{i}]} f( s_{ k}) \right) m(ds)\\
	 & = & \int_{0}^{ \infty }r^{ \zeta }E  \left[ I_{[  r^{ \beta }V \le 1]}\tilde g \Bigg(c_{ f}   \sum\limits_{ i=1}^{ | j |}\lambda _{ i} \Big( U\big(  t_{ i}- r^{ \beta}V\big) -U\big(  t_{ i-1}- r^{ \beta }V \big) \Big) \Bigg) \right]dr< \infty,
\end{eqnarray*}
where $    U( t)=\inf\{  x: S_{ \alpha }( x)\ge t  \},0\le t\le 1,  $ (we understand $ U( t)=0, \forall t< 0$) is the inverse time $ \alpha -$stable subordinator with
\[ 
 	 E  \left\{ \exp \left(- \lambda  S_{ \alpha }( 1) \right) \right\}= \exp  \left\{ - \frac{ \lambda ^{ \alpha }}{\Gamma( 1+\alpha ) } \right\} \ \ \ \ \forall \lambda \ge 0,   
\]
$ V $  is a random variable independent of $ (   U( t):0\le t\le 1  )$ with  distribution $ Q( \cdot)$ and 
\[ 
 	\tilde g(x) =\left\{    \begin{array}{ll} \sigma^{ 2}x^{ 2}/2  & \mbox{for  Theorem \ref{thm:moderate}} \\ g(x) & \mbox{for  Theorem \ref{thm:main}.} \end{array}  \right.   
\]
\end{lemma}

\begin{proof}
Define the probability measures $ m_{ l}( \cdot)=m( \cdot\cap S_{ l})/m( S_{ l})$ on $ (S,\mathcal{S})$, where $ S_{ l}=\{   s\in S:s_{ 0}\in E_{ l}  \}$. It is easy to check that  $ m(S_{ l})=\pi(E_{ l}).$ We begin by making a simple observation:
\begin{align*}
	 &    \frac{ 1}{ b_{ n}} \int_S  g \left( \frac{ b_{ n}}{a_{ n} }  \sum_{ i=1}^{ | j |}\lambda _{ i} \sum_{k=[nt_{i-1}]+1}^{[nt_{i}]} f( s_{ k}) \right) m(ds)\\
	 &  = \frac{1}{b_{ n}}\sum_{l=0}^{ \infty } m(S_{l}) \int_{S_{l}}  g \Big( \frac{ 1}{\gamma ( n) c_{ n}} \sum\limits_{ i=1}^{ | j |}\lambda _{ i} \sum\limits_{k=[nt_{i-1}]+1}^{[nt_{i}]} f( s_{ k}) \Big) m_{l}(ds)\\
	  & = \frac{c_{ n}^{ 2}}{\psi ^{ \leftarrow }( n)}\sum_{l=0}^{ \infty }\frac{\pi(S_{l})}{ \pi( S_{[  \psi^{ \leftarrow }( n)]})} \int_{S_{l}}  g \Big( \frac{ 1}{\gamma ( n) c_{ n}} \sum\limits_{ i=1}^{ | j |}\lambda _{ i} \sum\limits_{k=[nt_{i-1}]+1}^{[nt_{i}]} f( s_{ k}) \Big) m_{l}(ds)\\
	  & =  \int_{ 0}^{ \infty }h_{ n}( r)dr,
\end{align*}
where  for every $ r> 0$
\begin{align*} 
 	h_{ n}( r)&=    \frac{c_{ n}^{ 2} \pi(S_{[ r \psi^{ \leftarrow }( n)]})}{ \pi( S_{[  \psi^{ \leftarrow }( n)]})} \int\limits_{S_{[ r\psi^{ \leftarrow }( n)]}}  g \Big( \frac{ 1}{\gamma ( n)c_{ n} } \sum\limits_{ i=1}^{ | j |}\lambda _{ i} \sum\limits_{k=[nt_{i-1}]+1}^{[nt_{i}]} f( s_{ k}) \Big) m_{l}(ds)\\
	&=   \frac{ c_{ n}^{ 2}\pi(S_{[ r \psi^{ \leftarrow }( n)]})}{ \pi( S_{[  \psi^{ \leftarrow }( n)]})} E_{\pi_{[ r\psi^{ \leftarrow }( n)]}} \Bigg[  g \Big( \frac{ 1}{\gamma ( n)c_{ n} } \sum\limits_{ i=1}^{ | j |}\lambda _{ i} \sum\limits_{k=[nt_{i-1}]+1}^{[nt_{i}]} f( Z_{ k}) \Big)\Bigg].
\end{align*}
Under the conditions of Theorem \ref{thm:moderate}, $ c_{ n}$ is as described in the statement of the theorem and  in the case of Theorem \ref{thm:main} $ c_{ n}=1$.
We will prove the lemma in two steps. First we will show that 
\[  \lim_{ n\to \infty }h_{ n}( r)   = h(r):=r^{ \zeta} E  \left[ I_{[  r^{ \beta }V \le 1]}\tilde g \Bigg(c_{ f}   \sum\limits_{ i=1}^{ | j |}\lambda _{ i} \Big( U\big(  t_{ i}- r^{ \beta}V\big) -U\big(  t_{ i-1}- r^{ \beta }V \big) \Big) \Bigg) \right]
\]
and then in the second step we will prove that $ \int_{ 0}^{  \infty}h_{ n}(r)dr\to \int_{ 0}^{  \infty}h(r)dr$.

\textit{Step 1.}
Fix any  $r>0$
\begin{align*}
	 &   \lim_{ n \to \infty }c_{ n}^{ 2} \int_{S_{[r\psi^{ \leftarrow }( n )]}}  g \Big( \frac{ 1}{\gamma ( n)c_{ n} } \sum\limits_{ i=1}^{ | j |}\lambda _{ i} \sum\limits_{k=[nt_{i-1}]+1}^{[nt_{i}]} f( s_{ k}) \Big) m_{[r\psi^{ \leftarrow }(n)]}(ds)\\
	 & =  \lim_{ n \to \infty } c_{ n}^{ 2}E_{\pi_{[r\psi^{ \leftarrow }( n )]}}  \left[ g \Big( \frac{ 1}{\gamma ( n) c_{ n}} \sum\limits_{ i=1}^{ | j |}\lambda _{ i} \sum\limits_{k=[nt_{i-1}]+1}^{[nt_{i}]} f( Z_{ k}) \Big)\right]\\
	  &  = \lim_{ n \to \infty } c_{ n}^{ 2}E_{\pi_{[r\psi^{ \leftarrow }( n )]}}  \left[ g \Big( \frac{ 1}{\gamma ( n)c_{ n} }\sum\limits_{ i=1}^{ | j |}\lambda _{ i} \sum\limits_{k= [nt_{i-1}]\wedge T+1}^{[nt_{i}]\wedge T} f( Z_{ k}) + \frac{ 1}{\gamma ( n)c_{ n} } \sum\limits_{ i=1}^{ | j |}\lambda _{ i} \sum\limits_{k= [nt_{i-1}]\vee T+1}^{[nt_{i}]\vee T} f( Z_{ k})  \Big)\right]
\end{align*}
where for any sequence $ ( x_{ n})$ we understand $ \sum_{ k=i}^{ j}x_{ k}=0$ if $ j<i$. From assumption \eqref{eq:Gamma} it is easy to see that  
\begin{equation}\label{eq:frontpart}
	\frac{ 1}{\gamma ( n) }\sum\limits_{ i=1}^{ | j |}\lambda _{ i} \sum\limits_{k= [nt_{i-1}]\wedge T+1}^{[nt_{i}]\wedge T} f( Z_{ k}) \stackrel{ P_{ x}}{ \longrightarrow} 0 \ \ \ \ \mbox{ for all } x\in E.
\end{equation}
Next we concentrate on the second component.  Define the function $ \Psi:\mathcal{D}\to \mathcal{D}$ as 
\[ 
 	\Psi( h)( t):= \sum_{ i=1}^{ | j |  }\lambda _{ i}\Big( h\big((  t_{ i}-t)\vee 0 \big)-h\big( ( t_{ i-1}-t)\vee 0 \big)\Big)  , \ \ \ \ \mbox{ for all }    h\in \mathcal{D}, t\in [ 0,1].
\]
Note  that 
\[ 
 	\frac{ 1}{\gamma ( n) }\sum_{ i=1}^{ | j |}\lambda _{ i} \sum_{ k= [ nt_{ i-1}]+1}^{ [ nt_{ i}]} f( Z_{ k})  = \Psi( L_{ n})( 0),
\]
where 
\[ 
 	L_{ n}( t):= \frac{ 1}{ \gamma ( n)}\sum_{ k=1}^{ [ nt]}f( Z_{ k}), \ \ \ \ \mbox{ for all }   t\in [ 0,1].
\]
Since $ a$ is an atom $ T$ is independent of $ \sigma ( Z_{ n}:n\ge T)$ and   therefore for any measurable set $ A\subset \mathbb{R}$
\begin{eqnarray*}
	 &  & P_{\pi_{[r\psi^{ \leftarrow }( n )]}} \left[  \frac{ 1}{\gamma (n) }  \sum\limits_{ i=1}^{ | j |}\lambda _{ i} \sum\limits_{ k=[nt_{i-1}]\vee T+1}^{[nt_{i}]\vee T} f( Z_{ k})  \in A\right]\\
	 & = & P  \left[ \Psi( L_{ n})\Big( ( T^{ n}/n)\wedge 1 \Big) \in A\Big| Z_{ 0}\in a \right]
\end{eqnarray*}
where $ T^{ n}$ is a random variable independent of $ ( Z_{ n},n\ge 0)$ such that 
\[ 
	P[ T^{ n}\in \cdot]= P_{\pi_{[r\psi^{ \leftarrow }( n )]}}\big[  T\in \cdot\big].
\] 
Furthermore, if $ h$ is continuous and $ h_{ n}\to h$  in $ \mathcal{D}_{ Sk}$ then $ \big\| h_{ n}-h \big\|\to 0 $ and hence $ \big\| \Psi( h_{ n})-\Psi( h) \big\|\to 0$ which implies $ \Psi( h_{ n})\to \Psi( h)$ in $ \mathcal{D}_{ Sk}$.  By Lemma \ref{lem:MClimit} we know that $ L_{ n} \Longrightarrow c_{ f}U$ in $ \mathcal{D}_{ Sk}$, where  $ (   U( t):0\le t\le 1  )$ is as in the statement of this lemma. Since the $ 
(U( t):0\le t\le 1)$ is almost surely continuous  we can apply the continuous mapping theorem (see Theorem 2.7 in \cite{billingsley1999cpm}) to get 
\begin{equation}\label{eq:onelim}
	\Psi( L_{ n})\Longrightarrow \Psi( c_{ f}U) \ \ \ \ \mbox{ in }\mathcal{D}_{ Sk}.
\end{equation}
 Let $ c_{ n}$ be defined as  $ c_{ n}:=\psi( r\psi^{ \leftarrow }( n))$. Since $ \psi\in RV_{ \beta }$ it follows immediately that 
\[ 
 	\frac{ c_{ n}}{n }   \longrightarrow r^{ \beta } \ \ \ \ \mbox{ as } n \rightarrow \infty .
\]
By assumption \eqref{eq:psi} we get
\begin{equation}\label{eq:limitofT}
	\frac{ T^{ n}}{n }= \frac{ T^{ n}}{c_{ n} } \frac{ c_{ n}}{ n} \Longrightarrow r^{ \beta }V.
\end{equation}
Furthermore, since $ T^{ n}$ is independent of $ \{   Z_{ n},n\ge 0  \}$   we get 
\[ 
 	\Big( ( T^{ n}/n)\wedge 1, \Psi( L_{ n}) \Big)    \Longrightarrow \Big( r^{ \beta }V\wedge 1,\Psi( c_{ f}U)\Big) \ \ \ \ \mbox{  in } [ 0,1]\times \mathcal{D}.
\]
Also,  the map $ \tilde \psi: [ 0,1]\times \mathcal{D}\to \mathbb{R}$ defined by $ \tilde \psi( t,h)=h( t)$ is continuous at $ ( t,h)$ if $ h$ is continuous. Hence another application of the continuous mapping theorem gives us
\[ 
 	\Psi\big( L_{ n}\big)  \Big( ( T^{ n}/n)\wedge 1 \Big)   \Longrightarrow \Psi\big( c_{ f}U\big)\Big(  r^{ \beta }V\wedge 1\Big)
\]
Combining this with \eqref{eq:frontpart}  we obtain
\[ 
 	P_{\pi_{[r\psi^{ \leftarrow }( n )]}} \left[  \frac{ 1}{\gamma ( n) } \sum\limits_{ i=1}^{ | j |}\lambda _{ i} \sum\limits_{k=[nt_{i-1}]+1}^{[nt_{i}]} f( Z_{ k})  \in \cdot \right] \Longrightarrow P \left[  \Psi\big( c_{ f}U\big)\Big(  r^{ \beta }V\wedge 1\Big) \in \cdot\right].
\]
Under the conditions of Theorem \ref{thm:moderate} we know that $ g(x)\sim x^{ 2}\sigma ^{ 2}/2$ as $ x\to 0$ and hence
\begin{equation}\label{eq:gconvergence} 
	P_{\pi_{[r\psi^{ \leftarrow }( n )]}} \left[ c_{ n}^{ 2}g\Big( \frac{ 1}{\gamma ( n) c_{ n}} \sum\limits_{ i=1}^{ | j |}\lambda _{ i} \sum\limits_{k=[nt_{i-1}]+1}^{[nt_{i}]} f( Z_{ k}) \Big) \in \cdot \right] \Longrightarrow P \left[  \frac{ \sigma ^{ 2}}{2 }\Big(\Psi\big( c_{ f}U\big)\Big(  r^{ \beta }V\wedge 1\Big)\Big)^{ 2} \in \cdot\right].
\end{equation}
Under the setup of Theorem \ref{thm:main} we get 
\begin{equation}\label{eq:gconvergence:main} 
	P_{\pi_{[r\psi^{ \leftarrow }( n )]}} \left[ c_{ n}^{ 2}g\Big( \frac{ 1}{\gamma ( n) c_{ n}} \sum\limits_{ i=1}^{ | j |}\lambda _{ i} \sum\limits_{k=[nt_{i-1}]+1}^{[nt_{i}]} f( Z_{ k}) \Big) \in \cdot \right] \Longrightarrow P \left[  g\Big(\Psi\big( c_{ f}U\big)\Big(  r^{ \beta }V\wedge 1\Big)\Big) \in \cdot\right].
\end{equation}
Using \eqref{eq:gconvergence}, \eqref{eq:gconvergence:main} and Lemma \ref{lem:unifint} we get
\begin{align*}
	 &   \lim_{ n \to \infty } c_{ n}^{ 2}\int_{S_{[r\psi^{ \leftarrow }( n )]}}  g \Big( \frac{ 1}{\gamma ( n) c_{ n}} \sum\limits_{ i=1}^{ | j |}\lambda _{ i} \sum\limits_{k=[nt_{i-1}]+1}^{[nt_{i}]} f( s_{ k}) \Big) m_{[r\psi^{ \leftarrow }(n)]}(ds)\\
	 & =  E  \left[ \tilde g \left(\Psi\big( c_{ f}U\big)\Big(  r^{ \beta }V\wedge 1\Big)\right) \right]\\
	 & = E  \left[ I_{[  r^{ \beta }V \le 1]}\tilde g \Bigg(c_{ f}   \sum\limits_{ i=1}^{ | j |}\lambda _{ i} \Big( U\big(  t_{ i}- r^{ \beta}V\big) -U\big(  t_{ i-1}- r^{ \beta }V \big) \Big) \Bigg) \right].
\end{align*}
This combined with Assumption \eqref{eq:pi:regvar}   completes step 1 of the proof.

\textit{Step 2.} We will prove that the functions $ h_{ n}$ are dominated by an integrable function.  For that purpose note that it suffices to consider $ j=\{1     \}$ since 
\[
 	    g \Big( \frac{ 1}{\gamma ( n)c_{ n} } \sum\limits_{ i=1}^{ | j |}\lambda _{ i} \sum\limits_{k=[nt_{i-1}]+1}^{[nt_{i}]} f( Z_{ k}) \Big)
 	  \le  \max  \left\{ g\Big( \frac{ \bar{ \lambda }}{\gamma ( n) c_{ n}}\sum_{ k=1}^{ n}\big| f( Z_{ k}) \big|  \Big), g\Big( \frac{ -\bar{ \lambda }}{\gamma ( n)c_{ n} }\sum_{ k=1}^{ n}\big| f( Z_{ k}) \big|  \Big)  \right\} ,
\]
where $ \bar{ \lambda }=\max_{ 1\le k\le |j|}\big| \lambda _{ k} \big| .$  The following are the ingredients that we will use:
\begin{enumerate}[(i)] 
        \item  By Lemma \ref{lem:unifint} we know that  there exists $ K>0$ and $ N_{ 0}\ge 1$ such that 
\begin{equation}\label{eq:one} 
 	\sup_{ x\in E,n\ge N_{ 0}}E_{ x}\left[g \Big(  \frac{2 \lambda }{\gamma ( n)c_{ n} }\sum_{ k=1}^{ n} \big| f( Z_{ k}) \big|\Big)\right]\le K.	   
\end{equation}
\item Using \eqref{eq:Qassmp} and \eqref{eq:fassmp} we can get constants $ \epsilon ^{ \prime }>0$, $ c>0$,  $ N>1$, $ k_{ 0}>0$ and $ 0<\epsilon <\zeta +1$ such that for every $ r\ge 1$
\begin{equation}\label{eq:two} 
 	\sup_{ n\ge Nr} P_{ \pi_{ n}} \Bigg[  \frac{ T}{\psi( n) } \le cr^{ -\beta +\epsilon ^{ \prime }} \Bigg]    \le k_{ 0} r^{ -\zeta -1-\epsilon }
\end{equation}
and 
\begin{equation}\label{eq:three} 
 	\sup_{ n\ge Nr} E_{ \pi_{ n}}   \Bigg[ g\Big( \lambda \sum_{ i=1}^{ T\wedge \psi( n/r)} \big| f( Z_{ i}) \big| \Big) \Bigg]\le k_{ 0} r^{ -\zeta -1-\epsilon }.
\end{equation}
\item From Potter bounds (Proposition 1.5.6 in \cite{bingham:goldie:teugels:1989}or Theorem 1.5.6 in \cite{resnick:1987}) it is possible to get $ N_{ 1}>0$ and $ k_{ 1}>0$ such that 
\begin{equation}\label{eq:four} 
 	\sup_{ n\ge N_{ 1}}    \frac{ \pi(S_{[ r \psi^{ \leftarrow }( n)]})}{ \pi( S_{[  \psi^{ \leftarrow }( n)]})} \le\left\{\begin{array}{ll} k_{ 1} r^{ ( \zeta-\epsilon)\wedge 0 }& \mbox{ if } r\in ( 0,1)\\ k_{ 1}r^{ \zeta+\epsilon/2 } & \mbox{ if } r\ge 1.\end{array} \right. 
\end{equation}
\end{enumerate}

Combining \eqref{eq:one} and \eqref{eq:four} we get that  $ n\ge N_{ 1}$ implies 
\begin{equation}\label{eq:lostrack} 
 	h_{ n}( r)   \le Kk_{ 1} r^{ ( \zeta-\epsilon )\wedge 0} \ \ \ \ \mbox{ for all } r\in ( 0,1).
\end{equation}
For $ r\ge 1$ we use the convexity of $ g$ to get for  $ \lambda \in \mathbb{R}$
\begin{eqnarray*}
	 &  & E_{ \pi_{ [ r\psi^{ \leftarrow }( n)]}}  \left[ g \Big( \frac{ \lambda }{\gamma ( n) c_{ n}}\sum_{ k=1}^{ n} \big| f( Z_{ k}) \big| \Big)  \right]\\
	 &\le & \frac{ 1}{2 } E_{ \pi_{ [ r\psi^{ \leftarrow }( n)]}}  \left[ g \Big( \frac{2 \lambda }{\gamma ( n) c_{ n}}\sum_{ k=1}^{ T\wedge n} \big| f( Z_{ k}) \big| \Big)  \right]+ \frac{ 1}{2 }E_{ \pi_{ [ r\psi^{ \leftarrow }( n)]}}  \left[ g \Big(  \frac{2 \lambda }{\gamma ( n)c_{ n} }\sum_{ k=T\wedge n+1}^{ n} \big| f( Z_{ k}) \big|\Big)  \right]
\end{eqnarray*}
Then using \eqref{eq:three}  get $ N_{ 2}>0$ such that 
\[
	\sup_{n\ge N_{ 2} }E_{ \pi_{ [ r\psi^{ \leftarrow }( n)]}}  \Bigg[ g \Big( \frac{ 2 \lambda}{ \gamma ( n)c_{ n}} \sum_{ i=1}^{ T\wedge n} \big| f( Z_{ i}) \big| \Big)  \Bigg]\le k_{ 0}r^{ -\zeta -1-\epsilon }.
\]
For the second component observe that 
\[
	  E_{ \pi_{ [ r\psi^{ \leftarrow }( n)]}}  \left[ g \Big(  \frac{2 \lambda }{\gamma ( n)c_{ n} }\sum_{ k=T\wedge n+1}^{ n} \big| f( Z_{ k}) \big|\Big)  \right] 
	  \le 	  E_{ \pi_{ [ r\psi^{ \leftarrow }( n)]}}  \left[ I_{ [ T\le n]}g \Big(  \frac{2 \lambda }{\gamma ( n)c_{ n} }\sum_{ k=T+1}^{ T+n} \big| f( Z_{ k}) \big|\Big)  \right]
\]
and since $ a$ is an atom
\[
	  E_{ \pi_{ [ r\psi^{ \leftarrow }( n)]}}  \left[ I_{ [ T\le n]}g \Big(  \frac{2 \lambda }{\gamma ( n) c_{ n}}\sum_{ k=T+1}^{ T+n} \big| f( Z_{ k}) \big|\Big)  \right]
	  = P_{ \pi_{ [ r\psi^{ \leftarrow }( n)]}}  \left[  T\le n \right]E_{ a}\left[g \Big(  \frac{2 \lambda }{\gamma ( n)c_{ n} }\sum_{ k=1}^{ n} \big| f( Z_{ k}) \big|\Big)  \right]
\]
By another application of Potter bounds we can get  $ N_{ 3}>0$ such that
\[ 
 	\sup_{ n\ge N_{ 3}}\frac{ n}{ \psi\big( r\psi^{ \leftarrow }( n) \big)  }   \le c r^{ -\beta +\epsilon ^{ \prime }}
\] 
and this together with \eqref{eq:two} gives us that there exists $ N_{ 4}>0$ such that 
\begin{eqnarray*} 
 	\sup_{ n\ge N_{ 4}} P_{ \pi_{ [ r\psi^{ \leftarrow }( n)]}}\big[  T\le n \big] &  = &\sup_{ n\ge N_{ 4}} P_{ \pi_{ [ r\psi^{ \leftarrow }( n)]}}\Bigg[  \frac{ T}{\psi\big( r\psi^{ \leftarrow }( n) \big)  }\le \frac{ n}{\psi\big( r\psi^{ \leftarrow }( n) \big)  } \Bigg] \\
	& \le & k_{ 0}r^{ -\zeta -1-\epsilon}.
\end{eqnarray*}
Combining this with \eqref{eq:one} and \eqref{eq:lostrack}  we get  that $ h_{ n}( r)\le \tilde h( r)$ for all $ n\ge \max_{ 0\le i\le 4}N_{ i}$ and $ r>0$, where 
\[ 
 	\tilde h( r)= \left\{ \begin{array}{ll}  Kk_{ 1}r^{ ( \zeta-\epsilon )\wedge 0}  & \mbox{ if }r\in ( 0,1) \\ k_{ 1}k_{ 0}r^{  -1-\epsilon/2 }+Kk_{ 0}k_{ 1}r^{ -1-\epsilon /2} & \mbox{ if } r\ge 1\end{array}     \right.
\]
Observe that $\tilde h$ is integrable  because $ \zeta -\epsilon >-1$. Finally, we apply the dominated convergence theorem to get
\[ 
 	\lim_{ n \rightarrow \infty  } \int_{ 0}^{ \infty} h_{ n}(r)dr =\int_{ 0}^{ \infty}h(r)dr   
\]
and that completes the proof of the lemma.
\end{proof}

\begin{lemma}\label{lem:expttnsC:id}
If the conditions of  Theorem \ref{thm:moderate} or Theorem   \ref{thm:main} hold  then the family $\{\tilde \mu_n\}$ is exponentially  tight in $\mathcal{D}_S$, i.e., for every $p>0$ there exists a compact  $K_p\subset \mathcal{D}_S$, such that 
\[
\lim_{p\rightarrow \infty}\limsup_{n\rightarrow \infty }\frac{1}{b_n}\log \tilde \mu_n(K_p^c)=-\infty.
\]  
\end{lemma}

\begin{proof}
We use the notation $w(h, u ) :=  \sup\limits_{s,t\in[0,1],|s-t|< u }|h(s)-h(t)|$ for the modulus of  continuity of a function $h:[0,1]\rightarrow \mathbb{R}^d$. First we claim that for any $\epsilon>0,$  
\begin{equation}\label{mcont:id} 
	\lim_{ u \rightarrow 0}\limsup_{n\rightarrow \infty} \frac{1}{b_n}\log P\big(w(\tilde Y_n, u )>\epsilon\big)=-\infty, 
\end{equation} 
where $\tilde Y_n$ is the polygonal process in \eqref{e:polyg:id}.  Let us prove the lemma assuming that the claim is true.  By (\ref{mcont:id}) and the continuity of the paths of $\tilde Y_n$, there is $ u _k>0$ such that for all $n\geq 1$ 
\[
	P\big( w(\tilde Y_n, u _k)\ge k^{-1}  \big)\le e^{-p b_n   k}, 
\] 
and set $A_k=\{\xi\in \mathcal{D}:w(\xi, u _k)<k^{-1},\xi(0)=0\}.$ Now the set $K_p :=  \overline{\cap_{k\ge 1} A_k}$ is compact in $\mathcal{D}_S$ and by the union of events bound it follows that  
\[  
\limsup_{n\rightarrow \infty }\frac{1}{b_n}\log P(\tilde Y_n\notin K_p)\le -p,
\] 
establishing the exponential tightness. 

 Next we prove the claim (\ref{mcont:id}). Observe that for any $\epsilon>0$, $ u >0$ small, $ \lambda >0$ and $n>2/ u $ 
\begin{align*}
& P  \big( w(\tilde Y_n, u ) > \epsilon \big)   \le  P\Big( \max_{0\le i<j\le n,j-i\le  [n u ]+2}\frac{1}{a_n}\Big|\sum_{k=i}^jX_k\Big|>\epsilon \Big)\\ 
& \le  n\sum_{i=1}^{[2n u ]} P\Big(  \frac{b_n}{a_n}\Big|\sum_{k=1}^{i}X_k\Big|>b_n\epsilon\Big)\\ 
& \le  ne^{-b_n\lambda \epsilon} \sum_{i=1}^{[2n u ]}E\Big[ \exp  \Big( \frac{\lambda b_n}{a_n}\sum_{k=1}^iX_k\Big) +\exp \Big(  -\frac{\lambda b_n}{a_n}\sum_{k=1}^iX_k\Big) \Big]\\
& =  ne^{-b_n\lambda \epsilon} \sum_{i=1}^{[2n u ]}\exp\Big\{ \int_S  g \Big( \frac{\lambda  b_{ n}}{a_{ n} }   \sum_{k=1}^{i} f( s_{ i}) \Big) m(ds) \Big\} \\
&\ \ \ \  + ne^{-b_n\lambda \epsilon} \sum_{i=1}^{[2n u ]}\exp\Big\{ \int_S  g \Big( -\frac{\lambda  b_{ n}}{a_{ n} }   \sum_{k=1}^{i} f( s_{ i}) \Big) m(ds) \Big\}\Big).
\end{align*}
Now using the convexity of $g$ we get 
\begin{align*}
&  P  \big( w(\tilde Y_n, u ) > \epsilon \big) \\
& \le  \frac{4n^2 u }{e^{b_n\lambda \epsilon}}\exp\Big\{ \int_S  g \Big( \frac{\lambda  b_{ n}}{a_{ n} }   \sum_{k=1}^{[ 2 n  u  ]} |f( s_{ i})| \Big) m(ds) \Big\}\\
&   \ \ \ \ +  \frac{4n^2 u }{e^{b_n\lambda \epsilon}}\exp\Big\{ \int_S  g \Big( -\frac{\lambda  b_{ n}}{a_{ n} }   \sum_{k=1}^{[ 2 n  u  ]} |f( s_{ i})| \Big) m(ds) \Big\}.
\end{align*}
 Therefore by Lemma \ref{lem:id:limit} we have 
\[
 \lim_{ u \rightarrow 0}\limsup_{n\rightarrow    \infty}\frac{1}{b_n} \log P\big( w(\tilde Y_n, u )>\epsilon  \big)  \le  -\lambda \epsilon.
\]
Now, letting $\lambda\rightarrow \infty$ we obtain \eqref{mcont:id}. 
\end{proof}

\begin{lemma}\label{lem:ratefn:id} 
Assume that the conditions of  Theorem \ref{thm:moderate} or Theorem   \ref{thm:main} hold and let $ \Lambda_{ j}^{ *}$ be as defined in \eqref{eq:fidirtfn}. Then for any $ j=\{   0=t_{ 0}<t_{ 1}<\cdots<t_{ 1}\le 1  \}\in \mathcal{J}$ and any function $ \xi$ of bounded variation on $ [ 0,1]$ satisfying $ \xi( 0)=0$,
\[
	\sup_{ j\in \mathcal{J}} \Lambda_{ j}^{ *}\Big( \xi( t_{ 1}),\xi( t_{ 2})-f( t_{ 1}),\ldots, \xi( t_{ | j |})-\xi( t_{ | j |-1})\Big)
\]
\[ 
 	=I(\xi):= \left\{  \begin{array}{ll} \Lambda^{ *}_{ m }( \xi^{ \prime})   & \mbox{ if } \xi\in \mathcal{AC} \mbox{ and Theorem \ref{thm:moderate} holds}\\
	\Lambda^{ *}( \xi^{ \prime})   & \mbox{ if } \xi\in \mathcal{AC} \mbox{ and Theorem \ref{thm:main} holds}\\  \infty & \mbox{ otherwise.} \end{array}    \right.
\]
where $ \Lambda^{ *}_{ m }( \cdot)  $  is defined in \eqref{eq:ratefnform:id:mod} and $ \Lambda ^{ *}(\cdot)$ is defined in \eqref{eq:ratefnform:id}.
\end{lemma}

\begin{proof} 
First assume that $\xi\in \mathcal{AC}$. It is easy to see that the inequality 
\[ I(\xi) \ge \sup_{j\in \mathcal{J}}\Lambda^{*}_{j}(\xi(t_1),\xi(t_2)-\xi(t_1), \ldots,\xi(t_{|j|})-\xi( t_{ | j |-1}))\]
 holds  by considering a function $\psi\in L_\infty[0,1]$, which takes the value $ \lambda_i$ in the interval $(t_{i-1},t_i]$. 
 
 For the other inequality, take any $\psi\in L_{ \infty }[ 0,1]$ and choose  a sequence of uniformly bounded functions $\psi^n$ converging to $\psi$ almost everywhere  on $[0,1]$, such that for every $n$, $\psi^n$ is of the form $\sum_i\lambda_i^{n}I_{A_i^{n}},$ where $A_i^{n}=(t_{i-1}^{n},t_i^{n}]$, for some \[j^{ n}=\big\{0=t_{ 0}^{ n}<t_1^{n}<t_2^{n}<\cdots<t_{k_n}^{n}=1\big\}.\] Then by the continuity of $\tilde g$  and Fatou's Lemma, 
\begin{eqnarray*}
&& \int_{ 0}^{ 1} \psi( t)\xi^{ \prime }( t) dt -     \int_{0}^{ \infty }r^{ \zeta }E  \Big[ I_{[  r^{ \beta }V \le 1]}\tilde g \Big(c_{ f}   \int\limits_{0 }^{1- r^{ \beta }V } \psi( t) U( dt) \Big) \Big]dr\\
&=& \int_{ 0}^{ 1}\lim_{ n} \psi^{ n}( t)\xi^{ \prime }( t) dt -     \int_{0}^{ \infty }r^{ \zeta }E  \Big[ I_{[  r^{ \beta }V \le 1]}\tilde g \Big(c_{ f}   \int\limits_{0 }^{1- r^{ \beta }V }\lim_{ n} \psi^{ n}( t) U( dt) \Big) \Big]dr\\
&=& \lim_{ n} \int_{ 0}^{ 1}\psi^{ n}( t)\xi^{ \prime }( t) dt -     \int_{0}^{ \infty }r^{ \zeta }E  \Big[ I_{[  r^{ \beta }V \le 1]}\lim_{ n} \tilde g \Big(c_{ f}   \int\limits_{0 }^{1- r^{ \beta }V } \psi^{ n}( t) U( dt) \Big) \Big]dr\\
&\le& \lim_{ n} \int_{ 0}^{ 1}\psi^{ n}( t)\xi^{ \prime }( t) dt -   \limsup_{ n}  \int_{0}^{ \infty }r^{ \zeta }E  \Big[ I_{[  r^{ \beta }V \le 1]} \tilde g \Big(c_{ f}   \int\limits_{0 }^{1- r^{ \beta }V } \psi^{ n}( t) U( dt) \Big) \Big]dr\\
&=& \liminf_n\left\{ \sum_{i=1}^{k_n}\lambda_i^{n}\cdot \big( \xi(t_i^{n})-\xi(t_{i-1}^{n})\big)-\Lambda_{j^{ n}} (\lambda_1^{n},\cdots,\lambda_n^{n})\right\}\\
& \le & \sup_{j\in \mathcal{J}}\Lambda^{*}_{j}\big(\xi(t_1),\xi(t_2)-\xi(t_1), \ldots,\xi(t_{|j|})-\xi(t_{|j|-1})\big).
\end{eqnarray*}

Now suppose that $\xi$ is not absolutely continuous. That is, there exists $\epsilon>0$  and $0\le r_1^n<s_1^n\le r_2^n<\cdots \le r_{k_n}^n<s^n_{k_n}\le 1$, such that $\sum_{i=1}^{k_n}(s_i^n-r_i^n)\rightarrow 0$ but $\sum_{i=1}^{k_n}|\xi(s_i^n)-\xi(r_i^n)|\ge \epsilon$. Let $j^n$ be such that $t^n_{2p}=s^n_p$ and $t^n_{2p-1}=r^n_p$ (so that $|j^n|=2k_n$). Now   
\begin{eqnarray*}
&&  \sup_{j\in \mathcal{J}}\Lambda^{*}_{j}\big(\xi(t_1),\xi(t_2)-\xi(t_1), \ldots,\xi(t_{|j|})-\xi(t_{|j|-1})\big)\\ 
&\ge & \limsup_n \left\{\sup_{\underline{\lambda}^n\in \mathbb{R}^{2k_n}} \sum_{i=1}^{2k_n}\lambda_i^n\cdot \big(\xi(t_i^n)-\xi(t_{i-1}^n)\big) - \Lambda_{ j^{ n}}(\underline{\lambda}^n) \right\}\\ 
& \ge & \limsup_n \left\{ A \sum_{i=1}^{k_n}\big|\xi(s_i^n)-\xi(r_i^n)\big|-\Lambda_{j^{ n}}(\underline{\lambda}^{n*})\right\}\ge A\epsilon, 
\end{eqnarray*}
where $\lambda^{n*}_{2p-1}=0$ and $\lambda^{n*}_{2p}=A\big(\xi(s_i^n)-\xi(r_i^n)\big)/ |\xi(s_i^n)-\xi(r_i^n)|$ ($=0$ if $\xi(s_i^n)-\xi(r_i^n)=0$) for any $A>0$. The last inequality holds since $ \Lambda _{ j}( \underline{ \lambda} ^{ n*})\to 0$ as $ n \to \infty $, which follows from an application of dominated convergence theorem and  the fact that $ g$ is continuous at $0$ with $ g( 0)=0$.  This completes the proof since $A$ is arbitrary.  
\end{proof}

\begin{lemma}\label{lem:MClimit} 
Suppose $ f:E\to \mathbb{R}$ is  $ L_{ 1}( E,\mathcal{E},\pi)$ and $c_{ f}= \int_{ E}f( x)\pi( dx)\neq 0.$ Then for  any initial distribution $ \nu$ of $ Z_{ 0}$
\[ 
 	 \left( \frac{ 1}{\gamma ( n) }\sum_{ k=1}^{ [ nt]}f( Z_{ k}),t\in [ 0,1] \right)   \Longrightarrow  c_{ f}\Big( U( t),t\in [ 0,1] \Big) 
\]
in $ \mathcal{D}_{ Sk}$, where $ (U(t):0\le t\le 1)$ is as described in Theorem \ref{thm:moderate}.
\end{lemma}

\begin{proof} 
  This lemma is an extension of  Theorem 2.3 in \cite{chen:1999} which states that for any initial distribution $ \nu$ of $ Z_{ 0}$
  \[ 
 	\frac{ 1}{ \gamma ( n) }\sum_{ k=1}^{ n}f( Z_{ k}) \Longrightarrow c_{ f} U( 1).  
\]
 We proceed in a way similar to the proof of that theorem. By a  well known ratio limit theorem (see e.g. Theorem 17.3.2 in \cite{meyntweedie:1993}) we know that if $ g_{ 1},g_{ 2}\in L_{ 1}(  E, \mathcal{E}, \pi)$ with $ \int g_{ 2}( x)\pi( dx)\neq 0$ then 
 \[ 
 	\lim_{ n \rightarrow \infty   }  \sum\limits_{ k=1}^{ n}g_{ 1}( Z_{ k})\Big/ \sum\limits_{ k=1}^{ n}g_{ 2}( Z_{ k})  = \frac{ \int g_{ 1}( x) \pi( dx)}{ \int g_{ 2}( x)\pi( dx) }   
\]
Therefore it suffices to consider the function $ f:E^{ \prime }\to \mathbb{R}$ as $ f( x)=I_{ a}( x)$. 

Now, suppose $ I_{ n}=\sum_{k=1 }^{ n}f( Z_{ k})=\sum_{k=1 }^{ n}I_{ [ Z_{ k}\in a ]}$. By Theorem 2.3 in \cite{chen:1999} we get that  for any $ j=\{   0<t_{ 1}<\cdots <t_{ | j |}\le 1  \}\in \mathcal{J}$ and $ ( x_{ 1},\ldots, x_{ |j|})\in \mathbb{R}^{ |j|}$
\begin{eqnarray*}
	 &  & P_{ \nu} \big[  ( I_{ [ nt_{ 1}]},\ldots,I_{ [ nt_{ |j|}]})\le \gamma ( n)( x_{ 1},\ldots, x_{ |j|})  \big]\\
	 & = & P_{ \nu} \big[  ( T_{ [ \gamma ( n)x_{ 1}]},\ldots, T_{ [ \gamma ( n)x_{ |j|}]} ) \ge ( [ nt_{ 1}], \ldots , [ nt_{ |j|}])\big]   \\
	 & \sim & P_{ \nu} \Big[  \frac{ 1}{\gamma ^{ \leftarrow } ( k)}\big( T_{ [ kx_{ 1}]},\ldots , T_{ [ kx_{ |j|}]} \big)\ge  \frac{ 1}{\gamma ^{ \leftarrow } ( k)} \big( [ \gamma ^{ \leftarrow }( k)t_{ 1}],\ldots, [ \gamma ^{ \leftarrow }( k)t_{ |j|}] \big)  \Big] \\
	 & \to & P \Big[  \big( S_{ \alpha }( x_{ 1}),\ldots, S_{ \alpha }( x_{ |j|}) \big)  \ge ( t_{ 1}, \ldots, t_{ |j|})\Big] \\
	 & = & P \Big[ \big(  U( t_{ 1}),\ldots, U( t_{ |j|}) \big)\le ( x_{ 1},\ldots, x_{ |j|})\Big] 
\end{eqnarray*}
Therefore,
\[ 
 	\left( \frac{ 1}{\gamma ( n) }   I_{ [ nt_{ i}]}, i=1,\ldots |j|  \right) \Longrightarrow \Big( U( t_{ i}) , i=1,\ldots, |j|\Big)  , 
\]
which in turn implies
\begin{equation}\label{eq:findim} 
  	\left( \frac{ 1}{\gamma ( n) }   \sum_{ k=1}^{ [ nt_{ i}]}f( Z_{ k}), i=1,\ldots |j|  \right) \Longrightarrow c_{ f} \Big( U( t_{ i}) , i=1,\ldots, |j|\Big) .  
\end{equation}
We now need to prove tightness in the space $ \mathcal{D}_{ Sk}$. For that purpose consider the polygonal process
\[ 
 	\tilde L_{ n}( t)= \frac{ 1}{ \gamma ( n)}   \Big( \sum_{ k=1}^{ [ nt]}f( Z_{ k})+( nt-[ nt])f( Z_{ [ nt]+1}) \Big) \ \ \ \ \mbox{ for all }t\in [ 0,1].
\]
Let $ w( h,u )=\sup\limits_{ s,t\in[ 0,1],| s-t |<u }| h( s)-h( t) |$, be the modulus of continuity of a function $ h:[ 0,1] \rightarrow \mathbb{R}.$ Note that if suffices to prove that for any $ \epsilon >0$
\begin{equation}\label{eq:equicont:id}
	\lim_{ u \rightarrow 0 } \limsup_{ n \rightarrow \infty   } P_{ \nu}\big[  w( \tilde L_{ n},u)>\epsilon  \big] =0.
\end{equation}
For that purpose observe  that 
\begin{eqnarray*}
	 &&P_{ \nu}\big[  w( \tilde L_{ n},u)>\epsilon  \big] \\ & \le  & P_{ \nu} \Big( \max_{ 0\le i<j\le n,j-i\le [ nu ]+2} \frac{ 1}{\gamma ( n)}\Big| \sum_{ k=i}^{ j}f( Z_{ k})\Big|>\epsilon  \Big) \\
	 & \le &P_{ \nu} \Big( \max_{ 0\le i<j\le n,j-i\le [ nu ]+2} \frac{ 1}{\gamma ( n)}\sum_{ k=i}^{ j} \big| f( Z_{ k})\big|>\epsilon  \Big)\\
	 & \le &P_{ \nu} \Big( \max_{ 0\le i\le n-[ nu ]-2} \frac{ 1}{\gamma ( n)}\sum_{ k=i}^{ i+[ nu]+2} \big| f( Z_{ k})\big|>\epsilon  \Big)	 
\end{eqnarray*}
It is easy to check that for any non-decreasing function $ h:\mathbb{R}_{ +}\to \mathbb{R}_{ +}$ and $ u\in [ 0,1]$
\[ 
 	\sup_{ 0\le t\le 1-u} \big\{ h( t+u)-h( t)\big\} \le 2\max_{ 1\le i\le [ 1/u]+1}  \big\{ h( iu)-h( ( i-1)u) \big\}  
\]
which implies 
\begin{eqnarray*}
	 &  & P_{ \nu} \Big( \max_{ 0\le i\le n-[ nu ]-2} \frac{ 1}{\gamma ( n)}\sum_{ k=i}^{ i+[ nu]+2} \big| f( Z_{ k})\big|>\epsilon  \Big)\\
	 & \le & P_{ \nu} \Big( \max_{ 1\le i\le [ 1/u]+1}\sum_{ k=( i-1)[ nu]+1}^{ i[ nu]}\big| f( Z_{ k}) \big|>\epsilon /2  \Big) \\
	 & \le & \Big( \frac{ 1}{u }+1 \Big) \sup_{ x}P_{ x} \Big( \frac{ 1}{\gamma ( n) }\sum_{ k=1}^{ [ nu]} \big| f( Z_{ k}) \big| >\epsilon /2 \Big) 
\end{eqnarray*}
Now observe that 
\begin{eqnarray*}
	 &  &  \sup_{ x}P_{ x} \Big( \frac{ 1}{\gamma ( n) }\sum_{ k=1}^{ [ nu]} \big| f( Z_{ k}) \big| >\epsilon /2 \Big)\\
	 & \le & \sup_{ x}P_{ x} \Big( \frac{ 1}{\gamma ( n) }\sum_{ k=1}^{ [ nu]\wedge T} \big| f( Z_{ k}) \big|>\epsilon /4  \Big) +P_{ a}\Big( \frac{ 1}{\gamma ( n) }\sum_{ k=0}^{[ nu] \vee T-T } \big| f( Z_{ k}) \big|>\epsilon /4  \Big) \\
	 & \le & \sup_{ x}P_{ x} \Big( \frac{ 1}{\gamma ( n) }\sum_{ k=1}^{  T\wedge  n} \big| f( Z_{ k}) \big|>\epsilon /4  \Big) +P_{ a}\Big( \frac{ 1}{\gamma ( n) }\sum_{ k=1}^{[ nu]  } \big| f( Z_{ k}) \big|>\epsilon /4  \Big) 
\end{eqnarray*}
Again by assumption \eqref{eq:fassmp} it we get that 
\[ 
 	\limsup_{ n \rightarrow \infty  }    \sup_{ x}P_{ x} \Big( \frac{ 1}{\gamma ( n) }\sum_{ k=1}^{  T\wedge n} \big| f( Z_{ k}) \big|>\epsilon /4  \Big) =0,
\]
and by \eqref{eq:findim} 
\[ 
 	   \limsup_{ n \rightarrow \infty   } P_{ a}\Big( \frac{ 1}{\gamma ( n) }\sum_{ k=0}^{[ nu]  } \big| f( Z_{ k}) \big|>\epsilon /4  \Big) =P\big( c_{ | f |}U( u)>\epsilon /4 \big) ,
\]
where $ c_{ |f|}=\int_{ E}|f( s|)\pi( ds).$ Therefore,
\begin{eqnarray*}
	&& \limsup_{ n \rightarrow \infty   } P_{ \nu}\big[  w( \tilde L_{ n},u)>\epsilon  \big]\\
	 &\le  & \limsup_{ n \rightarrow \infty   }\Big( \frac{ 1}{u }+1 \Big) \sup_{ x}P_{ x} \Big( \frac{ 1}{\gamma ( n) }\sum_{ k=1}^{ [ nu]} \big| f( Z_{ k}) \big| >\epsilon /2 \Big)\\
	 &\le & \Big( \frac{ 1}{u }+1 \Big) P\big( c_{ | f |}U( u)>\epsilon /4 \big)\\
	 & = & \Big( \frac{ 1}{u }+1 \Big)P \big( S_{ \alpha }( \epsilon /4c_{ | f |})\le u \big)  
\end{eqnarray*}
Finally by Theorem 2.5.3 in \cite{zolotarev:1986} we get 
\[ 
	\lim_{ u \rightarrow 0 } \limsup_{ n \rightarrow \infty   } P_{ \nu}\big[  w( \tilde L_{ n},u)>\epsilon  \big] =0 	   
\] and that completes the proof of the lemma.
\end{proof}

\begin{lemma}\label{lem:unifint} 
If the conditions of Theorem \ref{thm:moderate} or Theorem \ref{thm:main} hold then for any $ j=\{ 0=t_{ 0}<t_{ 1}<\cdots<t_{ | j |}\le 1\}\in \mathcal{J}$ and $ \lambda _{ 1},\ldots,\lambda _{ | j |  }\in \mathbb{R}$ there exists $ K\ge 1$ such that 
 \[
	\sup_{ x\in E,n\ge K} E_{ x}  \left[ g \Big( \frac{ 1 }{\gamma ( n) c_{ n}} \sum\limits_{ i=1}^{ | j |}\lambda _{ i} \sum\limits_{k=[nt_{i-1}]+1}^{[nt_{i}]} f( Z_{ k})  \Big)I_{ \big[ g \big( \frac{ 1 }{\gamma ( n)c_{ n} } \sum\limits_{ i=1}^{ | j |}\lambda _{ i} \sum\limits_{k=[nt_{i-1}]+1}^{[nt_{i}]} f( Z_{ k}) \big)\ge N\big]} \right] \longrightarrow 0
\]
as $ N \to \infty $.
\end{lemma}
\begin{proof} 
It suffices to prove that for any $ \lambda \in \mathbb{R}, $
 \[
	\sup_{ x\in E,n\ge 1} E_{ x}  \left[ g \Big( \frac{ \lambda  }{\gamma ( n) c_{ n}} \sum\limits_{k=1}^{n} f( Z_{ k})  \Big)I_{ \big[ g \big( \frac{ \lambda  }{\gamma ( n)c_{ n} }  \sum\limits_{k=1}^{n} f( Z_{ k}) \big)\ge N\big]} \right] \longrightarrow 0
\]
as $ N\to \infty $. For that purpose, we look at 
\begin{align*} 
	&  P_{x} \left[  g \Big( \frac{ \lambda }{\gamma (n)c_{ n} }\sum_{i=1}^n f(Z_i) \Big)>t \right]\\
	& \le  P_{x} \left[ \Big(  \frac{  |\lambda| }{ \gamma ( n)c_{ n}}\sum_{i=1}^n \big|f(Z_i)\big|\Big)^{ \delta } >  (  \bar{g}(t))^{ \delta } \right]\\
	& \le  P_{x} \left[ \Big( \frac{  |\lambda| }{ \gamma ( n)c_{ n}}\Big)^{ \delta }\Big( \sum_{i=1}^{ I_n+1}\sum_{k=T_{ i-1}+1 }^{T_{ i} } \big|f(Z_k)\big|\Big)^{ \delta } >  (  \bar{g}(t))^{ \delta } \right]
\end{align*}
where, as before, $ I_{ n}=\sum_{k=1 }^{ n}I_{ [ X_{ k}\in a ]}$. By applying Holder's inequality we get 
\begin{align}
	 &    P_{x} \left[  \Big( \frac{  |\lambda| }{ \gamma ( n)c_{ n}}\Big)^{ \delta }\Big( \sum_{i=1}^{ I_n+1}\sum_{k=T_{ i-1}+1 }^{T_{ i} } \big|f(Z_k)\big|\Big)^{ \delta } >  (  \bar{g}(t))^{ \delta } \right]\nonumber \\
	 & \le   P_{x} \left[ \Big( \frac{  |\lambda| }{ \gamma ( n)c_{ n}}\Big)^{ \delta }2^{ \delta -1}\left\{ \Big( \sum_{k=1 }^{T_{ 1} } \big|f(Z_k)\big|\Big )^{ \delta } +\Big( \sum_{i=2}^{ I_n+1}\sum_{k=T_{ i-1}+1 }^{T_{ i} } \big|f(Z_k)\big|\Big)^{ \delta }\right\} >  (  \bar{g}(t))^{ \delta } \right]\nonumber \\
	 & \le  P_{x} \left[  \Big( \frac{  |\lambda| }{ \gamma ( n)c_{ n}}\Big)^{ \delta }2^{ \delta -1} \Big( \sum_{k=1 }^{T_{ 1} } \big|f(Z_k)\big|\Big )^{ \delta } > (  \bar{g}(t))^{ \delta }/2\right]\nonumber \\
	 & \hspace{1cm}+\ \ P_{ x}\left[\Big( \frac{  |\lambda| }{ \gamma ( n)c_{ n}}\Big)^{ \delta }2^{ \delta -1}  \Big( \sum_{i=2}^{ I_n+1}\sum_{k=T_{ i-1}+1 }^{T_{ i} } \big|f(Z_k)\big|\Big)^{ \delta } >  (  \bar{g}(t))^{ \delta } /2\right] \nonumber
\end{align}
which yields
\begin{align}\label{eq:konoekjaigay}
	&  P_{x} \left[  g \Big( \frac{ \lambda }{\gamma (n)c_{ n} }\sum_{i=1}^n f(Z_i) \Big)>t \right]\nonumber \\
	 & \le  P_{x} \left[  \Big( \frac{  |\lambda| }{ \gamma ( n)c_{ n}}\Big)^{ \delta }2^{ \delta -1} \Big( \sum_{k=1 }^{T_{ 1} } \big|f(Z_k)\big|\Big )^{ \delta } > (  \bar{g}(t))^{ \delta }/2\right] \\
	 & \hspace{1cm}+\ \ P_{ x}\left[\Big( \frac{  |\lambda| }{ \gamma ( n)c_{ n}}\Big)^{ \delta }2^{ \delta -1}  \Big( \sum_{i=2}^{ I_n+1}\sum_{k=T_{ i-1}+1 }^{T_{ i} } \big|f(Z_k)\big|\Big)^{ \delta } >  (  \bar{g}(t))^{ \delta } /2\right] \nonumber
\end{align}

The Assumption \eqref{eq:Gamma} takes care of the first component of \eqref{eq:konoekjaigay}
\begin{align}\label{eq:part1:id}
	 &   P_{x} \left[ \Big( \frac{  |\lambda| }{ \gamma ( n)c_{ n}}\Big)^{ \delta }2^{ \delta -1} \Big( \sum_{k=1 }^{T_{ 1}  } \big|f(Z_k)\big|\Big )^{ \delta } > (  \bar{g}(t))^{ \delta }/2\right] \nonumber\\
	 & \le  \exp( -k_{ 0}\bar{ g}( t)^{ \delta })E_{ x}  \left[\exp\left\{ \frac{ 1}{k_{ 0} }\Big( \frac{ 2| \lambda  |}{ \gamma ( n)c_{ n}}\Big)^{ \delta }\Big( \sum_{k=1 }^{T_{ a} } \big|f(Z_k)\big|\Big )^{ \delta } \right\}\right]
\end{align}
where $ k_{ 0}$ is as in \eqref{eq:integ}.  

Next we consider the second component of \eqref{eq:konoekjaigay}. Another application of Holder's inequality gives us 
\begin{align}\label{eq:aaroekta}
	 &   P_{ x}\left[\Big( \frac{  |\lambda| }{ \gamma ( n)c_{ n}}\Big)^{ \delta }2^{ \delta -1}  \Big( \sum_{i=2}^{ I_n+1}\sum_{k=T_{ i-1}+1 }^{T_{ i} } \big|f(Z_k)\big|\Big)^{ \delta } >  (  \bar{g}(t))^{ \delta } /2\right]\\
	  & \le  P_{ a}\left[\Big( \frac{  |\lambda| }{ \gamma ( n)c_{ n}}\Big)^{ \delta }2^{ \delta -1}  \Big( \sum_{i=1}^{ I_n+1}\sum_{k=T_{ i-1}+1 }^{T_{ i} } \big|f(Z_k)\big|\Big)^{ \delta } >  (  \bar{g}(t))^{ \delta } /2\right]\nonumber\\
	 & \le  P_{ a}\left[ \Big( \frac{ 2| \lambda  |}{ \gamma ( n)c_{ n}}\Big)^{ \delta }(  I_{ n}+1)^{ \delta -1} \sum_{i=1}^{ I_n+1}\Big(\sum_{k=T_{ i-1}+1 }^{T_{ i} } \big|f(Z_k)\big|\Big)^{ \delta } >  (  \bar{g}(t))^{ \delta } \right]. \nonumber
\end{align}
and using the fact that $ I_{ n}\le n$ for every $ n$ we get the bound
\begin{align}
& P_{ a}\left[ \Big( \frac{ 2| \lambda  |}{ \gamma ( n)c_{ n}}\Big)^{ \delta }(  I_{ n}+1)^{ \delta -1} \sum_{i=1}^{ I_n+1}\Big(\sum_{k=T_{ i-1}+1 }^{T_{ i} } \big|f(Z_k)\big|\Big)^{ \delta } >  (  \bar{g}(t))^{ \delta } \right]\\
	  & \le  \sum_{ l=1}^{ (n+1)/\gamma ( n)} P_{ a}\left[\Big(\frac{   2| \lambda  |}{c_{ n}} \Big)^{ \delta }\frac{ l^{ \delta -1}}{ \gamma ( n)} \sum_{i=1}^{ l\gamma ( n)}\Big(\sum_{k=T_{ i-1}+1 }^{T_{ i} } \big|f(Z_k)\big|\Big)^{ \delta } >  (  \bar{g}(t))^{ \delta },  l-1< \frac{ I_{ n}+1}{ \gamma ( n)}\le l\right]. \nonumber
\end{align}
Applying Holder's inequality for the third time yields for every $ l\ge 1$
\begin{align}\label{eq:hoytoshesh}
	  &   P_{ a}\left[\Big(\frac{   2| \lambda  |}{c_{ n}} \Big)^{ \delta }\frac{ l^{ \delta -1}}{ \gamma ( n)}  \sum_{i=1}^{ l\gamma ( n)}\Big(\sum_{k=T_{ i-1}+1 }^{T_{ i} } \big|f(Z_k)\big|\Big)^{ \delta } >  (  \bar{g}(t))^{ \delta },  l-1< \frac{ I_{ n}+1}{ \gamma ( n)}\le l\right]\\
	  & \le  P_{ a}\left[\Big(\frac{   2| \lambda  |}{c_{ n}} \Big)^{ \delta }\frac{ l^{ \delta -1}}{ \gamma ( n)} \sum_{i=1}^{ l\gamma ( n)}\Big(\sum_{k=T_{ i-1}+1 }^{T_{ i} } \big|f(Z_k)\big|\Big)^{ \delta } >  (  \bar{g}(t))^{ \delta }\right]^{ 1/p} P_{ a}\left[  l-1< \frac{ I_{ n}+1}{ \gamma ( n)}\le l \right]^{ 1-1/p}\nonumber\\
	 & =R^{ 1/p}S^{ 1-1/p} \ \ \ \ \mbox{(Say)} \nonumber
\end{align}
where $ p>1$. We consider the parts $ R$ and $ S$ separately. To part $ R$ we  apply an exponential Markov inequality to obtain the bound
\begin{align*}
	 & R=  P_{ a}\left[ \Big(\frac{   2| \lambda  |}{c_{ n}} \Big)^{ \delta }\frac{ l^{ \delta -1}}{ \gamma ( n)}  \sum_{i=1}^{ l\gamma ( n)}\Big(\sum_{k=T_{ i-1}+1 }^{T_{ i} } \big|f(Z_k)\big|\Big)^{ \delta } >  (  \bar{g}(t))^{ \delta }\right]\\
	 & \le   l \gamma ( n) \Gamma \left(\Big(\frac{   2| \lambda  |}{c_{ n}} \Big)^{ \delta }\frac{ l^{ \delta -1}}{ \gamma ( n)}  \right) \exp\left( -k_{ 0}p\bar{ g}( t)^{ \delta }\right) 
\end{align*}
where 
\[ 
 	\Gamma( t ):=\log E_{ a}  \Big[ \exp  \Big(  t \sum_{k=1 }^{ T}\big| f( Z_{ k}) \big|  \Big)^{ \delta } \Big]   .
\]
Since  $ \delta <( 1-\alpha )^{ -1}$ and  $ l$ is atmost $ n/\gamma ( n)$ for any $ n\ge 1$ and hence  there exists $ K_{ 1}>0$ such that 
\[ 
 	\frac{ l^{ \delta -1}}{\gamma ( n) }   \le \frac{ n^{ \delta -1}}{\gamma ( n)^{ \delta } }\le K_{ 1} \ \ \ \ \mbox{ for all } n\ge 1.
\]
Now observe that  $ \Gamma $ is convex and $ \Gamma ( 0)=0$. Under the conditions of Theorem \ref{thm:main}, $ \Gamma (t)< \infty$ for all $ t>0$ and  we can get $ K_{ 2}>0$ such that 
\[ 
 	\Gamma ( t)\le K_{ 2}t \ \ \ \ \mbox{ for all }   0\le t\le (  2| \lambda  |)^{ \delta }K_{ 1}/k_{ 0}p.
\]
If the conditions of Theorem \ref{thm:moderate} hold then there exists $ t_{ 0}>0$ such that $ \Gamma (t)< \infty$ for all $ 0<t\le t_{ 0}$. In this case also there exists $ K_{ 2}>0$ such that 
\[ 
 	\Gamma ( t)\le K_{ 2}t \ \ \ \ \mbox{ for all }   0\le t\le t_{ 0},
\]
and furthermore there exists $ K\ge 1$ such that $ \Big(\frac{   2| \lambda  |}{c_{ n}} \Big)^{ \delta }K_{ 1}\le t_{ 0} $ for all $ n\ge K$.
Therefore, under the conditions of Theorem \ref{thm:moderate} or Theorem \ref{thm:main} for any $ n\ge K$ and $ l\le n/\gamma (n)$ we have 
\begin{eqnarray}\label{eq:partn}
	 &  & P_{ a}\left[ \Big(\frac{   2| \lambda  |}{c_{ n}} \Big)^{ \delta }\frac{ l^{ \delta -1}}{ \gamma ( n)} \sum_{i=1}^{ l\gamma ( n)}\Big(\sum_{k=T_{ i-1}+1 }^{T_{ i} } \big|f(Z_k)\big|\Big)^{ \delta } >  (  \bar{g}(t))^{ \delta }\right] \nonumber \\
	 & \le &   l \gamma ( n)K_{ 2}\frac{ 1}{k_{ 0}p }\Big(\frac{   2| \lambda  |}{c_{ n}} \Big)^{ \delta }\frac{ l^{ \delta -1}}{ \gamma ( n)}  \exp( -k_{ 0}p\bar{ g}( t)^{ \delta }) \nonumber \\
	 & = & K_{ 2} \frac{ 1}{k_{ 0}p } \Big( \frac{ 2| \lambda  |}{c_{ n}}\Big)^{ \delta } l^{ \delta } \exp( -k_{ 0}p\bar{ g}( t)^{ \delta }) .
\end{eqnarray}
Now we concentrate on  part $ R$. We use the duality between the variables $ (T_{ n})$ and $ (I_{ n})$  to get 
\[
	 P_{ a}  \left[ I_{ n} \ge l\gamma ( n) \right] =  P_{ a}  \left[ T_{ [ l \gamma ( n)]} \le n \right]	  \le  P_{ a}  \left[ T_{ l[ \gamma ( n)]} \le n\right]. 
\]
Define the variables $ W_{ k}:=T_{ k}-T_{ k-1}$ for $ k\ge 1.$ Since $ a$ is an atom, $ (   W_{ k} )$ is a sequence of i.i.d. random variables. For any $ x>1/n$
\begin{eqnarray*}
	P\big[ W_{ 1}+\cdots+W_{ [ \gamma ( n)]}\le nx\big] & \le & P\left [ \max_{ 1\le i\le [ \gamma ( n)]}W_{ i}\le nx\right]\\
	&= &  \left( 1- \frac{ 1}{b\pi( C)\gamma ( nx) } \right)^{ [ \gamma ( n)]}
\end{eqnarray*}
There exists $ c_{ 1}>0$  such that for any $ n\ge 1$ and $ 1/n<x\le 2,$
\begin{equation}
		P\big[ W_{ 1}+\cdots+W_{ [ \gamma ( n)]}\le nx\big] \le    \exp \left( -c_{ 1} \frac{ [ \gamma ( n)]}{\gamma ( nx) } \right)
\end{equation}
Fix $ \epsilon >0$ such that $ \alpha -\epsilon >0$. Using Potter bounds (Theorem 1.5.6 in \cite{bingham:goldie:teugels:1989}) we  get $ c_{ 2}>1$ such that for $ x_{ 1}>x_{ 2}>1$
\begin{equation}
	 c_{ 2} \left( \frac{ x_{ 1}}{x_{ 2} } \right)^{ \alpha -\epsilon  }\le  \frac{ \gamma ( x_{ 1})}{\gamma( x_{ 2}) } .
\end{equation}
Hence, it is easy to get $ c_{ 3}>0$ such that for all $ 1/n<x\le 2,$
\begin{equation}\label{eq:headprob}
	P\big[ W_{ 1}+\cdots+W_{ [ \gamma ( n)]}\le nx\big] \le   \exp \left( -c_{ 3} x^{ \alpha -\epsilon } \right).
\end{equation}
Now if $ V^{ n}_{ k}:=( W_{ ( k-1)[ \gamma ( n)]+1}+\cdots+W_{ k[ \gamma ( n)]})/n$ then 
\[ 
 	P \big[  T_{ l[ \gamma ( n)]}\le n \big]   = P \big[  V^{ n}_{ 1}+\cdots+V^{ n}_{ l}\le 1 \big]. 
\]
By \eqref{eq:headprob} there exists $ \sigma >0$ such that for any $ 0<x\le 2$
\[ 
 	P  \Big[ \frac{ 1}{n }+ V^{ n}_{ l}\le x \Big]  \le P \big[\sigma   S_{ \alpha -\epsilon }\le x \big] 
\]
where $ S_{ \alpha -\epsilon }$ is a right skewed $ ( \alpha -\epsilon )$-stable random variable satisfying 
\[ 
 	   E \big[  \exp ( -t S_{ \alpha -\epsilon }) \big] = \exp ( -t^{ \alpha -\epsilon }) \ \ \ \ \mbox{ for all } t>0.
\]
Using the fact that $l\le n/\gamma ( n) $ for any $ n\ge 1$
\begin{eqnarray*}
	 P \big[  T_{ l[ \gamma ( n)]}\le n \big] &  = &P \big[  V^{ n}_{ 1}+\cdots+V^{ n}_{ l}\le 1 \big] \\
	 & \le & P \Big[\Big(  \frac{ 1}{n }+V^{ n}_{ 1}\Big)  +\cdots + \Big( \frac{ 1}{n }+V^{ n}_{ l} \Big) \le 2 \Big] \\
	 & \le & P \big[  S_{ \alpha -\epsilon }l^{1(  \alpha -\epsilon) }\le 2/\sigma  \big] 
\end{eqnarray*}
By Theorem 2.5.3 in \cite{zolotarev:1986} there exists $ c_{ 4}>0$ and $ c_{ 5}>0$ such that 
\begin{equation}\label{eq:partlast}
	P\big[  T_{ l[ \gamma ( n)]}\le n \big] \le c_{ 4}\exp \left( -c_{ 5}l^{ \frac{ 1}{1-\alpha +\epsilon  }} \right).
\end{equation}
Therefore, by combining \eqref{eq:aaroekta}-\eqref{eq:partn} and \eqref{eq:partlast} we get 
\begin{eqnarray}\label{eq:partn:id}
	&  & P_{ x}\left[\frac{  |\lambda|^{ \delta } }{ \gamma ( n)^{ \delta }}2^{ \delta -1}  \Big( \sum_{i=2}^{ I_n+1}\sum_{k=T_{ i-1}+1 }^{T_{ i} } \big|f(Z_k)\big|\Big)^{ \delta } >  (  \bar{g}(t))^{ \delta } /2\right] \\
	& \le & \exp( -k_{ 0}\bar{ g}( t)^{ \delta }) \sum_{ l=1}^{ \infty } K_{ 2} \frac{ 1}{k_{ 0}p } (  2| \lambda  |)^{ \delta } l^{ \delta } c_{ 4}\exp \left( - c_{ 5}( 1-1/p)l^{  \frac{ 1}{1-\alpha +\epsilon  }}\right)\nonumber
\end{eqnarray}
The series in \eqref{eq:partn:id} surely  converges to a  finite number and hence by \eqref{eq:part1:id} and \eqref{eq:partn:id}  there exists a constant $ C>0$ such that 
\[
	  P_{x} \left[  g \Big( \frac{ \lambda }{\gamma (n)c_{ n} }\sum_{i=1}^n f(Z_i) \Big)>t \right] \le C  \exp( -k_{ 0}\bar{ g}( t)^{ \delta })
\]
Finally,
\begin{eqnarray*}
	&  & E_{x} \left[ g \Big( \frac{ \lambda }{\gamma (n) }\sum_{i=1}^n f(Z_i) \Big)I_{ \big[ g \big( \frac{ \lambda }{\gamma (n) }\sum\limits_{i=1}^n f(Z_i) \big)>N\big]}\right]\\
	& = & \int_N^\infty P_{x} \left[  g \Big( \frac{ \lambda }{\gamma (n) }\sum_{i=1}^n f(Z_i) \Big)>t \right]dt+NP_{x} \left[  g \Big( \frac{ \lambda }{\gamma (n) }\sum_{i=1}^n f(Z_i) \Big)>N \right]
\end{eqnarray*}
and that converges $ 0 $ as $ N\to \infty $ by  Assumption \eqref{eq:integ}.
\end{proof}

\end{section}

\begin{section}{Examples of Long Range Dependent ID Processes}\label{sec:exam} 

 \begin{example}[Simple symmetric random walk on $ \mathbb{Z}$]\label{eg:rw1dim} 
 Suppose $ E=\mathbb{Z}$ and $ \mathcal{E}$ is the power set of $ E$. Let $ ( Z_{ n})$ be the simple symmetric random walk on $ \mathbb{Z}$, i.e., it is a markov chain with transition kernel 
 \[ 
 	p( i,j)=\left\{  \begin{array}{ll}  1/2  &\mbox{if } j=i+1 \mbox{  or } j=i-1\\ 0 & \mbox{otherwise.} \end{array}  \right.  
\] 
Then the  counting measure $ \pi$ on $(  E,\mathcal{E})$ is an invariant measure for this  kernel $ p( \cdot,\cdot)$. Here, we can take $ E_{ n}=\{   -n,n  \}$ and $ E_{ 0}={ 0},$ which means $ \zeta=0.$ Furthermore, $ a=\{   0  \}$ is an atom for $ ( Z_{ n}).$ From the arguments proving Proposition 2.4 in \cite{legallrosen:1991} we get
\[ 
 	\gamma ( n)\sim \sum_{ k=1}^{ n}P_{ 0}[ X_{ k}=0] \sim \sqrt{ n}  \sqrt{ 2/\pi}\in RV_{ 1/2}. 
\]
and hence $ \alpha =1/2$. Also it is easy to observe that 
\[ 
 	   P_{ \pi_{ n}}\big[  T \in \cdot \big] \stackrel{ d}{ =} \sum_{ k=1}^{ n} W_{ k} 
\]where $ (W_{ k})$ are i.i.d. with distribution $ P[W_{ 1}\in \cdot]= P_{ \pi_{ 1}}[T\in \cdot]$. So the hitting time of $ a$ from $ E_{ n}$ can be expressed as a sum of $ n$ i.i.d. random variables. Using this fact it is also easy to check the well known result
\[ 
 	P_{ \pi_{ n}}\big[  T/n^{ 2}\in \cdot \big] \Rightarrow S_{ 1/2},    
\]
where $ S_{ 1/2}$ is a right-skewed $ 1/2$-stable distribution with density
\[ 
 	h( x)=\frac{ 1}{\sqrt{ 2\pi} }   x^{ -3/2}\exp \big\{ -( 2x)^{ -1} \big\}, \ \ \ \ \forall x>0.
\]Therefore, $ \psi( n)=n^{ 2}$, $ \beta =2$ and $ Q( \cdot)$ is the law of $ S_{ 1/2}$. By the arguments used to prove the statement \eqref{eq:partlast} and the representation of the distribution $ T$ under $ P_{ \pi_{ n}}$ it is possible to check that  assumption \eqref{eq:Qassmp} is satisfied. Now suppose that $ \rho( \cdot)$ is a L\'evy measure on $ ( \mathbb{R},\mathcal{B}( \mathbb{R}))$ such that 
\[ 
 	g( \lambda )=\int_{ \mathbb{R}}\big( e^{ \lambda z}-1 -\lambda \llbracket z \rrbracket\big)    \rho( dz)< \infty ,\ \ \ \ \forall \lambda \in \mathbb{R},
\]
and 
\begin{equation}\label{eq:ex1:d} 
 	\int_{ 0}^{ \infty} \exp \big( -k_{ 0}\bar g( t)^{ \delta } \big)dt< \infty ,    
\end{equation}
for some $ \delta <2$ and $ k_{ 0}>0$ where $ \bar{ g}( t)=\min \{   | s |:g( s)=t  \}.$ Let $ M$ be an IDRM on $ S$ with control measure $ m$ defined in \eqref{eq:contrmeas} and local characteristics $ (0,\rho,0)$.

\textit{Moderate Deviation Principle:} Suppose $ f:\mathbb{Z}\to \mathbb{R}$ is a function that satisfies conditions F and Assumption \ref{assmp:integ:g:f}.  Define the ID process $ X_{ n}=\int_{ S}f(s_{ n})M(ds)$ and assume that $ E(X_{ 1})=0$ (this is an assumption on $ \rho$ and $ f$) and $ var(X_{ 1})=\sigma ^{ 2}$. If \eqref{eq:ex1:d}  holds with  $ \delta =1$  then any function with bounded support satisfies those conditions, but if $ \delta >1$ then the only choice for $ f$ is of the form $ cI_{ 0}(x)$ for some $ c\ne 0$. Suppose $ c_{ n}\to \infty$ is such that $ \sqrt{n}/c_{ n}^{ 2}\to \infty$ and let $ \mu_{ n}$ be the law of
\[ 
 	Y_{ n}( t)= \frac{ c_{ n}}{n }\sum_{ i=1}^{ [ nt]}X_{ i}, \ \ \ \ t\in [ 0,1].   
\]
in $ \mathcal{BV}$. Then $ (   \mu_{ n}  )$ satisfies LDP in $ \mathcal{BV}$ with speed $ \sqrt{ n}/c_{ n}^{ 2}$ and good rate function 
\begin{equation}
	H_{ m}( \xi)= \left\{  \begin{array}{cc}  \Lambda^{ *}_{ m}( \xi^{ \prime })   & \mbox{ if } \xi\in \mathcal{AC}, \xi( 0)=0\\ \infty & \mbox{ otherwise.} \end{array} \right.
\end{equation}
where for any $ \varphi\in L_{ 1}[ 0,1]$
\begin{eqnarray}
	 \Lambda^{ *}_{ m}( \varphi)&=  & \sup_{ \psi\in L_{ \infty }[ 0,1]}\left\{ \int_{ 0}^{ 1} \psi( t)\varphi( t) dt \right.   \nonumber \\
	 & & \left.-  \int_{0}^{ \infty }2E  \Big[ I_{[  r^{ 2 }S^{ *}_{ 1/2} \le 1]} \frac{ \sigma ^{ 2}}{2 } \Big(c_{ f}\sqrt{ \frac{ \pi}{ 2}}   \int_{ 0}^{1- r^{ 2 }S^{ *}_{ 1/2} } \psi( t) U( dt) \Big)^{ 2} \Big]dr \right\}.
\end{eqnarray}
Here  $    U( t):=\inf\{  x: S_{ 1/2 }( x)\ge t  \},0\le t\le 1,  $ is the inverse time $ 1/2$-stable subordinator, where
\[ 
 	 E  \left\{ \exp \left(- \lambda  S_{ 1/2 }( 1) \right) \right\}= \exp  \left\{ - \frac{2 }{\sqrt{ \pi} } \lambda ^{ 1/2 }\right\},\forall \lambda \ge 0,   
\]
and $ S_{ 1/2}^{ *} $  is independent of $ \{   U( t):0\le t\le 1  \}$ having the same distribution as $ S_{ 1/2}( 1)$.

\textit{Large Deviation Principle:} The only functions $ f$ which satisfy the conditions of Theorem \ref{thm:main} is of the form $ cI_{ 0}(x)$ where $ c\ne 0$. Then the law of 
\[ 
  	Y_{ n}( t)= \frac{ 1}{n }\sum_{ i=1}^{ [ nt]}X_{ i}, \ \ \ \ t\in [ 0,1].   
\]
satisfies LDP in $ \mathcal{BV}$ with speed $ \sqrt{n}$ and good rate function 
\begin{equation}
	H( \xi)= \left\{  \begin{array}{cc}  \Lambda^{ *}( \xi^{ \prime })   & \mbox{ if } \xi\in \mathcal{AC}, \xi( 0)=0\\ \infty & \mbox{ otherwise.} \end{array} \right.
\end{equation}
where for any $ \varphi\in L_{ 1}[ 0,1]$
\begin{eqnarray}
	 \Lambda^{ *}( \varphi)&=  & \sup_{ \psi\in L_{ \infty }[ 0,1]}\left\{ \int_{ 0}^{ 1} \psi( t)\varphi( t) dt \right.   \nonumber \\
	 & & \left.-  \int_{0}^{ \infty }2E  \Big[ I_{[  r^{ 2 }S^{ *}_{ 1/2} \le 1]} g \Big(c_{ f}\sqrt{ \frac{ \pi}{ 2}}   \int_{ 0}^{1- r^{ 2 }S^{ *}_{ 1/2} } \psi( t) U( dt) \Big) \Big]dr \right\}.
\end{eqnarray}
 \end{example}

 \begin{example}\label{eg:hittime} 
 Suppose that  $ (Z_{n})$ is a Markov chain on $ E=\mathbb{Z}_{+}$ with transition probabilities
 \begin{equation}\label{eq:trans}
	P(i,j)= \left\{ \begin{array}{ll}  p_{ i}q_{ i}  & \text{ if } i\neq 0, j=i+1\\ p_{ i}( 1-q_{ i}) & \text{ if }i\neq 0, j=0 \\  1-p_{ i} &  \text{ if }j=i\\ 1 & \text{ if }i=0,j=1\\ 0 & \text{otherwise}\end{array} \right.
\end{equation}
where $ ( p_{ n})$ and $ ( q_{ n})$ are two sequences of real numbers between $ 0$ and $ 1$ ($ p_{ 0}=1, q_{ 0}=1$). The Markov chain $ ( Z_{n} )$ is  irreducible and if $\prod_{ k=0}^{ \infty }q_{ n}=0 $ then $ (Z_{ n})$ is recurrent.  Whenever the chain hits a state $i$, it stays there for $\tau_{i}$ amount of time where $ \tau_{i}\sim$ geometric($ p_{ i}$). When it leaves $i$,  it jumps to $i+1$ with probability $ q_{ i}$ or goes to $ 0$ with probability $ 1-q_{ i}$.  Therefore, given that $X_{0}=0$, we can write
\[ 
 	T\stackrel{d}{=}1+ \sum_{k=1}^{N}\tau_{k},  
\] 
where $ N$ is a random variable independent of  $ (\tau_{ k},k\ge 1)$  and having distribution 
\[ 
 	P[ N=n]= q_{ 1}\cdots q_{n-1 }( 1-q_{ n}), \ \ \ \ \mbox{ for every } n\ge 1.   
\]
Clearly that means $ P_{0}[ T< \infty  ]=P_{ 0}[ N< \infty ]= 1$ if  $  \prod_{ k=0}^{ \infty }q_{ n}=0.$ In that case it is also easy to check that
\[
	 \pi(  0)=1 \mbox{ and }  \pi(  n)=q_{ 1}\cdots q_{ n-1}/p_{ n}, \ \ \ \ \mbox{ for all } n\ge 1
\] 
is an invariant measure for this Markov chain.

 Consider a model where $ (p_{ n})$ and $ (q_{ n})$ are of the form 
\[ 
 	p_{ n}= \frac{ 1}{( n+1)^{ s} } \ \ \ \ \mbox{ and } \ \ \ \  q_{ n}= \Big( \frac{ n}{n+1 }\Big)^{ t} \ \ \ \ \mbox{ for every } n\ge 1.   
\]
Assume $ s>0$ and $ t>0$ satisfies $ 1/2<t-s<1$.  In this setup it is easy to check that $ (Z_{ n})$ is a recurrent Markov chain. Also observe that
\begin{equation}\label{eq:dist:N} 
 	P\big[  N=n \big]\sim     \frac{ t}{n^{ t+1} }\in RV_{ -( t+1)} \ \ \ \ \mbox{and} \ \ \ \  P\big[  N> n \big]   \sim \frac{ 1}{  n^{ t}} \in RV_{ -t}.
\end{equation}
Therefore
\[ 
 	E_{ 0}[ T] = 1+\sum_{ n=1}^{ \infty }P[ N=n] \sum_{ k=1}^{ n} \frac{ 1}{p_{ k} }= \infty,
\]
 which means that $ (Z_{ n})$ is a null-recurrent Markov chain. We can take $ a=\{   0  \}$ to be an atom. Next we find $ \alpha $ by estimating the tail probability of the random variable $ T$. Note that
\[ 
 	T\stackrel{ d}{ =}1+\sum_{ k=1}^{ N}\tau_{ k}= \sum_{ k=1}^{ N}\big( \tau_{ k}-(k+1)^{ s}\big) +h( N)
\]
where $ h$ is defined as 
\[ 
 	h( n):=1+ \sum_{ k=1}^{ n} ( k+1)^{ s} \sim \frac{ 1}{s+1 }n^{ s+1}\in RV_{ s+1}.   
\]
From  \eqref{eq:dist:N} it follows that  $ P[ h( N)>n]\sim \big( ( s+1) n\big)^{-t/( s+1) }\in RV_{- t/( s+1)}$. If we can show that $ \sum_{ k=1}^{ N}( \tau_{ k}-( k+1)^{ s})$ has a lighter tail then it would follow that $ P[ T>n]\in RV_{ -t/( s+1)}$ which would in turn imply $ \alpha =t/( s+1)$. For that purpose observe that $ t/( s+1)<1<t/( s+1/2)$.  Then by Cauchy Schwartz inequality
\begin{eqnarray*}
	E \left[\Big|  \sum_{ k=1}^{ N}\Big( \tau_{ k}-( k+1)^{ s}\Big) \Big|\right]& \le  &  E  \left[  \sum_{ k=1}^{ N}E\Big( \tau_{ k}-( k+1)^{ s} \Big)^{ 2}  \right]^{ 1/2}
\end{eqnarray*}
Since $ Var( \tau_{ k})=( k+1)^{ 2s}$ we get 
\begin{eqnarray*}
	E \left[ \Big|  \sum_{ k=1}^{ N}\Big( \tau_{ k}-( k+1)^{ s}\Big)\Big|\right] &\le   & E  \left[  \sum_{ k=1}^{ N}( k+1)^{ 2s} \right]^{ 1/2}\\
	& \le &cE  \left[  N^{ 2s+1} \right] ^{ 1/2}\\
	& = &c E  \left[ N^{ s+1/2} \right]< \infty .
\end{eqnarray*}
where $ c>0$ is a constant such that 
\[ 
 	\sum_{ k=1}^{ n}( k+1)^{ 2s}\le c n^{ 2s+1} \ \ \ \ \mbox{  for all }n\ge 1.   
\]
Therefore we have
\[ 
 	\gamma( n) =   ( s+1)  n^{ t/( s+1)}.
\]
A natural partition in to take take $ E_{ n}=\{   n  \}$ for every $ n\ge 0$ and 
\[ 
 	\pi( n)= \frac{ ( n+1)^{ s}}{n^{ t} }\in RV_{ s-t}   
\]
which means $ \zeta =s-t>-1.$ Next we  claim that $ \beta=s+1$. Note that given $ X_{ 0}=n$
\[ 
 	T \stackrel{ d}{= } \sum_{ k=0}^{ N_{ n}}\tau_{ n+k}    
\]
where $ N_{ n}$ is a random variable such that for $ k\ge 1$
\[ 
 	P   \big[  N_{ n} =k \big] = q_{ n}\cdots q_{ n+k-1}( 1-q_{ n+k}) =P \big[  N=n+k \big|N\ge n\big] 
\]
It follows immediately that 
\begin{equation}\label{eq:Nn}
	P\big[  N_{ n}/n >x \big] \sim \frac{ 1}{( 1+x)^{ t} } \ \ \ \ \mbox{ for all }x>0.
\end{equation}
Following the same argument as above, using Cauchy Schwartz inequality it is easy to check that 
\begin{equation}\label{eq:centered} 
 	E  \left|\frac{ 1}{n^{ s+1} }  \sum_{ k=0}^{ N_{ n}}\Big( \tau_{ n+k}-( n+k)^{ s}\Big) \right|   \longrightarrow 0.
\end{equation}
Therefore for any $ x>0$
\begin{eqnarray*}
	\lim_{ n \rightarrow \infty   } P_{ \pi_{ n}}  \left[ T/n^{ s+1} >x \right]  & = & \lim_{ n \rightarrow \infty   } P  \left[ \frac{ 1}{ n^{ s+1}} \sum_{ k=0}^{ N_{ n}}( n+k)^{ s}>x\right]\\
	& = & \lim_{ n \rightarrow  \infty } P\left[ \frac{ 1}{n^{ s+1} }\Big(  h ( n+N_{ n})-h( n-1)\Big)>x\right]\\
	& = & \lim_{ n \rightarrow \infty   }P  \left[ \Big( 1+ \frac{ N_{ n}}{n } \Big)^{ s+1}>1+( s+1)x  \right] \\
	& = & \big( 1+( s+1)x \big)^{- \frac{ t}{s+1 }}.
\end{eqnarray*}
Hence $ \psi( n)=n^{ s+1}$ and $ Q( \cdot)$ is a measure on $ ( \mathbb{R}_{ +},\mathcal{B}( \mathbb{R}_{ +}))$ such that 
\[ 
 	Q( ( x, \infty ))   =\big( 1+( s+1)x \big)^{- \frac{ t}{s+1 }} \ \ \ \ \mbox{ for all } x>0.
\]
Next we ensure that assumption \eqref{eq:Qassmp} is satisfied. Fix any $ 0<\epsilon <t$. Then for any $ r\ge 1$
\begin{eqnarray*}
	  P_{ \pi_{ n}}  \left[ \frac{ T}{n^{ s+1} }\le cr^{ -s-1+\epsilon } \right]
	 & = & P  \left[ \frac{ 1}{n^{ s+1} }\sum_{ k=0}^{ N_{ n}}\tau_{ n+k}\le c r^{ -s-1+\epsilon }\right]\\
	 & \le & P  \left[ \frac{ 1}{n^{ s+1} }\sum_{ k=0}^{ N_{ n}^{ \prime}}W_{ k}\le c r^{ -s-1+\epsilon } \right]
\end{eqnarray*}
where $ (W_{ k}, k\ge 0)$ are i.i.d geometric($ p_{ n}$) and $ N_{ n}^{ \prime }$ is geometric($ 1-q_{ n}$) and is independent of $ (W_{ k})$. Now 
\[ 
 	 \sum_{ k=0}^{ N_{ n}^{ \prime}}W_{ k} \sim \text{ geometric}(p_{ n}( 1-q_{ n}))   
\]
and therefore it is possible to get $ c^{ \prime }>0$  and $ K>1$ such that
\begin{eqnarray*}
	\sup_{ n\ge Kr}  P  \left[ \frac{ 1}{n^{ s+1} }\sum_{ k=0}^{ N_{ n}^{ \prime}}W_{ k}\le c r^{ -s-1+\epsilon } \right]
	 & \le & 1-\exp  \big\{ c^{ \prime } cr^{ -s-1+\epsilon }n^{ s+1}p_{ n}( 1-q_{ n}) \big\}.
\end{eqnarray*}
Observe that $n^{ s+1}p_{ n}( 1-q_{ n})\to t $ as $ n\to \infty $. Hence one can get $ c^{ \prime \prime }>0$  such that 
\[ 
 	\sup_{ n\ge Kr}  P  \left[ \frac{ 1}{n^{ s+1} }\sum_{ k=0}^{ N_{ n}^{ \prime}}W_{ k}\le c r^{ -s-1+\epsilon } \right]\le c^{ \prime \prime }r^{ -s-1+\epsilon }.   
\]
Since $ \zeta=s-t$ and we assumed that $ t>\epsilon $, we see that \eqref{eq:Qassmp} is satisfied.  

Now suppose that $ \rho( \cdot)$ is a L\'evy measure on $ ( \mathbb{R},\mathcal{B}( \mathbb{R}))$ such that 
\[ 
 	g( \lambda )=\int_{ \mathbb{R}}\big( e^{ \lambda z}-1 -\lambda \llbracket z \rrbracket\big)    \rho( dz)< \infty ,\ \ \ \ \forall \lambda \in \mathbb{R},
\]
and \eqref{eq:ex1:d} holds for some  $ \delta <( s+1)/( s-t+1)$.  Let $ M$ be an IDRM on $ S$ with control measure $ m$ defined in \eqref{eq:contrmeas} and local characteristics $ (0,\rho,0)$.

\textit{Moderate Deviation Principle:} Suppose $ f:\mathbb{Z}\to \mathbb{R}$ is a function that satisfies conditions F and Assumption \ref{assmp:integ:g:f}.  Define the ID process $ X_{ n}=\int_{ S}f(s_{ n})M(ds)$ and assume that $ E(X_{ 1})=0$  and $ var(X_{ 1})=\sigma ^{ 2}$. Here also if \eqref{eq:ex1:d}  holds with  $ \delta =1$  then any function with bounded support satisfies those conditions, but if $ \delta >1$ then the only choice for $ f$ is of the form $ cI_{ 0}(x)$ for some $ c\ne 0$. From the computations above we see that 
\[ 
 	\pi( E_{ [ \psi^{ \leftarrow }( n)]})\gamma ( n)\psi^{ \leftarrow }( n)\sim ( s+1)n \ \ \ \ \mbox{and} \ \ \ \    \pi( E_{ [ \psi^{ \leftarrow }( n)]})\psi^{ \leftarrow }( n)\sim n^{ \frac{ s-t+1}{s+1 }}.
\]

Suppose $ c_{ n}\to \infty$ is such that $ n^{ \frac{ s-t+1}{s+1 }}/c_{ n}^{ 2}\to \infty$ and let $ \mu_{ n}$ be the law of
\[ 
 	Y_{ n}( t)= \frac{ c_{ n}}{n }\sum_{ i=1}^{ [ nt]}X_{ i}, \ \ \ \ t\in [ 0,1].   
\]
in $ \mathcal{BV}$. Then $ (   \mu_{ n}  )$ satisfies LDP in $ \mathcal{BV}$ with speed $  n^{ \frac{ s-t+1}{s+1 }}/c_{ n}^{ 2}$ and good rate function 
\begin{equation}
	H_{ m}( \xi)= \left\{  \begin{array}{cc}  \Lambda^{ *}_{ m}( \xi^{ \prime })   & \mbox{ if } \xi\in \mathcal{AC}, \xi( 0)=0\\ \infty & \mbox{ otherwise.} \end{array} \right.
\end{equation}
where for any $ \varphi\in L_{ 1}[ 0,1]$
\begin{eqnarray}
	 \Lambda^{ *}_{ m}( \varphi)&=  & \sup_{ \psi\in L_{ \infty }[ 0,1]}\left\{ \int_{ 0}^{ 1} \psi( t)\varphi( t) dt \right.   \nonumber \\
	 & & \left.-  \int_{0}^{ \infty }2E  \Big[ I_{[  r^{ 2 }S^{ *}_{ 1/2} \le 1]} \frac{ \sigma ^{ 2}}{2 } \Big( \frac{ c_{ f}}{s+1 }   \int_{ 0}^{1- r^{ 2 }S^{ *}_{ 1/2} } \psi( t) U( dt) \Big)^{ 2} \Big]dr \right\}.
\end{eqnarray}
Here  $    U( x):=\inf\{  y: S_{ t/( s+1) }( y)\ge x  \},0\le x\le 1,  $ is the inverse time $ t/( s+1)$-stable subordinator and $V $  is independent of $ (   U( x):0\le x\le 1  )$ having distribution $ Q( \cdot)$.

\textit{Large Deviation Principle:} The only functions $ f$ which satisfy the conditions of Theorem \ref{thm:main} is of the form $ cI_{ 0}(x)$. Then the law of 
\[ 
  	Y_{ n}( t)= \frac{ 1}{n }\sum_{ i=1}^{ [ nt]}X_{ i}, \ \ \ \ t\in [ 0,1].   
\]
satisfies LDP in $ \mathcal{BV}$ with speed $ n^{ \frac{ s-t+1}{s+1 }}$ and good rate function 
\begin{equation}
	H( \xi)= \left\{  \begin{array}{cc}  \Lambda^{ *}( \xi^{ \prime })   & \mbox{ if } \xi\in \mathcal{AC}, \xi( 0)=0\\ \infty & \mbox{ otherwise.} \end{array} \right.
\end{equation}
where for any $ \varphi\in L_{ 1}[ 0,1]$
\begin{eqnarray}
	 \Lambda^{ *}( \varphi)&=  & \sup_{ \psi\in L_{ \infty }[ 0,1]}\left\{ \int_{ 0}^{ 1} \psi( t)\varphi( t) dt \right.   \nonumber \\
	 & & \left.-  \int_{0}^{ \infty }2E  \Big[ I_{[  r^{ 2 }S^{ *}_{ 1/2} \le 1]} g \Big( \frac{ c_{ f}}{s+1 }  \int_{ 0}^{1- r^{ 2 }S^{ *}_{ 1/2} } \psi( t) U( dt) \Big) \Big]dr \right\}.
\end{eqnarray}

\end{example}

\end{section}

\section*{Acknowledgements} The author is thankful to Prof. Gennady Samorodnitsky for his comments, numerous helpful discussions and reading through a previous draft of the paper in details. The author is thankful to Prof. Jan Rosinski for teaching a course on Infinitely Divisible processes at Cornell University and for many helpful discussions on this topic. The material described in Section \ref{sec:backgr} is derived from that course.

\bibliographystyle{imsart-nameyear}
\bibliography{/Users/Souvik/Documents/Chronicles/Miscellanea/Latex/bibfile}
\end{document}